\documentclass{article}

\usepackage[text={6.2in, 7.9in},centering]{geometry}
\usepackage{amsmath,xcolor,cite}
\usepackage{bm,amssymb,mathrsfs,url,units,dsfont}
\usepackage{setspace,bbold}
\usepackage{graphicx,cite,mathtools}
\usepackage{amssymb,mathrsfs,paralist,bm,esint,setspace}
\RequirePackage{amsfonts,amsthm,amsmath,color}
\usepackage{utopia}
\definecolor{ForestGreen}{RGB}{20,140,80}
\RequirePackage[colorlinks,citecolor=ForestGreen,urlcolor=blue,linkcolor=blue]{hyperref}

\usepackage{amssymb}
\usepackage{esint,comment}
\usepackage{mathrsfs}
\usepackage{mathtools}
\usepackage{hyperref}
\usepackage{amsthm}
\usepackage{amsmath}
\usepackage{bm}
\usepackage{enumerate}
\usepackage{tikz-cd,soul}
\usepackage[normalem]{ulem}
\usepackage{bbold}
\usepackage{indentfirst}
\usepackage{titlesec}

\titlelabel{\thetitle.\,\,}

\numberwithin{equation}{section}
\theoremstyle{plain}

\newtheorem{theorem}{Theorem}[section]
\newtheorem{lemma}[theorem]{Lemma}

\newtheorem{remark}[theorem]{\bf{Remark}}
\newtheorem{assumption}[theorem]{Assumption}

\newtheorem{definition}[theorem]{Definition}

\newtheorem{proposition}[theorem]{Proposition}

\theoremstyle{remark}
\theoremstyle{definition}

\newcommand\1{\mathds{1}}

\newcommand\dP{\mathds{P}}
\newcommand\dR{\mathds{R}}
\newcommand\dE{\mathds{E}}
\newcommand\dZ{\mathds{Z}}
\newcommand\dN{\mathds{N}}

\newcommand\bI{{\bf I}}

\newcommand\bM{{\bf M}}

\newcommand\bN{{\bf N}}

\newcommand\bJ{{\bf J}}

\newcommand\fM{\mathfrak{M}}

\newcommand\cF{\mathcal{F}}

\newcommand\cP{\mathcal{P}}

\newcommand\cS{\mathcal{S}}

\newcommand\cN{\mathcal{N}}

\newcommand\fA{\mathfrak{A}}

\newcommand\fr{\mathfrak{r}}

\newcommand\fJ{\mathfrak{J}}
\newcommand\fR{\mathfrak{R}}
\newcommand\fd{\mathsf{d}}
\newcommand\fg{\mathfrak{g}}

\newcommand\cbrk{\text{$]$\kern-.15em$]$}}
\newcommand\opar{\text{\,\raise.2ex\hbox{${\scriptstyle
|}$}\kern-.34em$($}}
\newcommand\cpar{\text{$)$\kern-.34em\raise.2ex\hbox{${\scriptstyle |}$}}\,}
\newcommand\supp{\mathrm{supp}}

\begin{document}

\title{On the support of solutions to nonlinear stochastic heat equations}

\author{Beom-Seok Han
\and Kunwoo Kim  
\and Jaeyun Yi 
}

\date{\today}

\maketitle

\begin{abstract}  
We investigate the strict positivity and the compact support property of solutions to the one-dimensional nonlinear stochastic heat equation: $$\partial_t u(t,x) = \frac{1}{2}\partial^2_x u(t,x) + \sigma(u(t,x))\dot{W}(t,x), \quad (t,x)\in \dR_+\times\dR,$$ with nonnegative and compactly supported initial data $u_0$, where $\dot{W}$ is the space-time white noise and $\sigma:\mathds{R} \to \mathds{R} $ is a continuous function with $\sigma(0)=0$. We prove that (i) if $v/ \sigma(v)$ is sufficiently large near  $v=0$, then the solution $u(t,\cdot)$ is strictly positive for all $t>0$, and (ii) if $v/\sigma(v)$ is sufficiently small near $v= 0$, then the solution $u(t,\cdot)$ has compact support for all $t>0$. These findings extend previous results concerning the strict positivity and the compact support property, which were analyzed only for the case $\sigma(u)\approx u^\gamma$ for $\gamma>0$. Additionally, we establish the uniqueness of a solution and the weak comparison principle in case (i).

\vspace{1cm} 
 
	
\end{abstract}


\section{Introduction}\label{sec: introduction}
We consider the one-dimensional stochastic heat equation (SHE) given by

\begin{equation}\label{eq:SHE}
  \begin{aligned}
    \begin{cases}
      \partial_t u(t,x) = \frac{1}{2}\partial^2_x u(t,x) + \sigma(u(t,x))\dot{W}(t,x), \quad(t,x)\in \dR_+ \times \dR,\\
      u(0,x)  = u_0(x), \quad x\in\dR,
    \end{cases}
  \end{aligned}
\end{equation}
where $\dot W$ represents space-time (Gaussian) white noise on $\dR_+\times \dR$, $u_0$ is a sufficiently smooth and compactly supported function, and $\sigma: \dR \to \dR$ is a continuous function with $\sigma(0)=0$. The principal goal  is to explore how the nonlinearity $\sigma(u)$ influences the support of solutions to \eqref{eq:SHE}. Specifically, we are interested in conditions on $\sigma$  under which  solutions to \eqref{eq:SHE}  exhibit either full support (i.e. the solution is strictly positive everywhere) or compact support with probability 1.

When $\sigma(u)$ is proportional to $u^\gamma$ near $u=0$ for $\gamma>0$, there is an extensive literature that  provides a detailed understanding of how the support of the solution is influenced by the value of $\gamma$.  In \cite{iscoe1988supports}, Iscoe was among the first to demonstrate that for  $\gamma=1/2$,  a solution to \eqref{eq:SHE} exhibits compact support for all $t>0$ with probability 1, a property we refer to as the Compact Support Property (CSP). Note that for $\gamma=1/2$, $u$ is considered as the density of super-Brownian motion and  Iscoe  used some properties of super-Brownian motion to show CSP. Following this, Shiga \cite{shiga1994two} extended Iscoe's methodology to establish CSP for \eqref{eq:SHE} with $\gamma \in (0, 1/2)$. On the other hand,  Mueller and Perkins \cite{mueller1992compact} considered the case where $\gamma \in (1/2, 1)$. This particular case is more delicate and interesting  because this case can be considered as a super-Brownian motion with reduced birth-death rates (i.e. the coefficient of the noise $u^\gamma$ is effectively less than $u^{1/2}$ near $u=0$). By creatively  constructing and employing a historical process, they proved that CSP remains valid in this delicate case. Diverging from these approaches, Krylov \cite{krylov1997result} used  a weak solution framework to \eqref{eq:SHE} to prove CSP for $\gamma \in (0, 1)$. In contrast, Mueller \cite{mueller1991support} investigated the case $\gamma \geq 1$, demonstrating that $u(t, x)$ remains strictly positive for all $t>0$ and $x\in \dR$ with probability 1, thereby indicating the absence of CSP in these cases. This background leads us to the following natural  questions: 
\begin{itemize}
\item  What outcomes can be expected  for $\sigma$ satisfying $u \ll \sigma(u) \ll u^\gamma$ when $u$ is near zero?  In particular, under what conditions can $\sigma$ ensure  CSP or guarantee strict positivity of solutions to \eqref{eq:SHE}?
\end{itemize}
Those questions have previously been addressed within the context of deterministic partial differential equations (PDEs).  More precisely, consider   \begin{equation}\label{eq:HE}
  \begin{aligned}
    \begin{cases}
      \partial_tu(t,x) = \frac{1}{2}\partial^2_x u(t,x) - \sigma(u(t,x)), \quad(t,x)\in \dR_+ \times \dR,\\
      u(0,x) = u_0(x), \quad x\in\dR. 
    \end{cases}
  \end{aligned}
\end{equation}
where  $u_0$ is a non-negative and continuous function with compact support, and  $\sigma:\dR \to \dR_+$ is a  non-decreasing and continuous  function  with $\sigma(0)=0$ and $\sigma(u)>0$ for all $u\in \dR\setminus\{0\}$.   In 1974, Kalashnikov \cite[Theorem 7]{kalashnikov1974propagation} proved that CSP holds  for  \eqref{eq:HE} if 
\begin{equation}\label{eq:condition for CSP of deterministic PDE}
  \int_0^1 \frac{du}{\sqrt{u\sigma(u)}} < +\infty. 
\end{equation}  
Furthermore,  in \cite[Theorem 11]{kalashnikov1974propagation}, strict positivity  of the solution to \eqref{eq:HE} is shown to hold    when 
\begin{equation}\label{eq:condition for strict positivity of deterministic PDE}
  \int_0^1 \frac{du}{\sigma(u)} = +\infty,
\end{equation} when $u_0(x)>0$ for all $x\in \dR$. On the other hand, \cite[Remark 2.4]{evans1979instantaneous} shows that if
\begin{equation}\label{eq:evans-condition}
\int_0^1 \frac{du}{\sqrt{u\sigma(u)}} = \infty,
\end{equation}
then there exist initial functions $u_0$ such that the solution $u(t, x)$ to \eqref{eq:HE} with initial condition $u(0, x) = u_0(x)$ is strictly positive for all $x \in \dR$ and all $t\in [0, t_0]$ for some $t_0>0$. Numerous studies have examined CSP and strict positivity in deterministic PDEs, establishing conditions for these properties (see, e.g. \cite{chen1995finite,evans1979instantaneous,galaktionov1994extinction,knerr1979behavior} for related results). However, to the best of our knowledge,  literature on conditions for CSP and strict positivity  for stochastic heat equations other than the case where $\sigma(u) \approx u^\gamma$ is extremely limited. This paper aims to bridge this gap by proposing specific  conditions on $\sigma$ that ensure  CSP or strict positivity. We now detail these conditions below.

\begin{assumption}\label{assumption:basic condition for sigma} 
  We assume that  $\sigma:\dR\to \dR_+$ is a nondecreasing continuous function with $\sigma(u)\1_{u\leq 0} = 0$ for $u\in \dR$. We further assume that there exist constants  $\fd>0$ and $L_\fd>0$  such that
  \begin{itemize}
        \item [(a)] For all $u,v>0$ with $u\leq v\leq \fd$, we have $0<\frac{\sigma(v)}{v}\leq\frac{\sigma(u)}{u}$,
        \item [(b)] $|\sigma(u)| \leq L_{\fd} |u| $ and $|\sigma(u)-\sigma(v) | \leq L_{\fd}|u-v|$ for all $u,v\geq \fd.$\label{item:third condition of sigma}
  \end{itemize} 
\end{assumption}

\begin{remark}
  The assumption that $\sigma(u)\1_{u\leq 0 }=0$ is somewhat technical and can be weakened. Throughout the paper, we only deal with nonnegative solutions to \eqref{eq:SHE} (see Theorem \ref{thm:weak comparison principle} and Theorem \ref{thm:cpt_support}). Therefore, we only need to define $\sigma(u)$ for $u\geq0$ and set $\sigma(0)=0$. 
\end{remark}

We note that Assumption \ref{assumption:basic condition for sigma}  guarantees the existence of a (\emph{mild and weak}) solution.  However, the uniqueness is a tricky problem (see \cite{burdzy2010nonuniqueness, mytnik2011pathwise,mytnik2006pathwise}).  Under some additional condition on $\sigma$, we can obtain the uniqueness of a mild solution and the strict positivity of the solution.

\begin{theorem}[\bf Well-posedness and Strict positivity]\label{thm:strict positivity}
Suppose $\sigma:\dR\to \dR$ satisfies Assumption \ref{assumption:basic condition for sigma} and there exist $\alpha\in(0,\frac{1}{4})$ and $v_1 = v_1(\alpha)>0$ such that for all $v\in (0,v_1)$ 
  \begin{equation}\label{eq:positivity_condition} 
    \frac{v}{\sigma(v)} \geq (-\log v)^{-\alpha}. 
  \end{equation}
  Let $u_0\in C(\dR)$ be bounded, nonnegative and not identically zero. Then, we have the following: 
\begin{itemize}
\item[(i)] There exists a unique mild solution $u$ to \eqref{eq:SHE}, and $u$ has a version which is space-time H\"older continuous of regularity $\left(\frac{1}{2}-, \frac{1}{4}-\right)$.
\item[(ii)] For any $T>0$, 
  \begin{equation*}
    \dP \left( u(t,x) >0 \text{ for all }(t,x)\in(0,T) \times \dR\right) =1.
  \end{equation*}
\end{itemize} 
\end{theorem}

To establish Theorem \ref{thm:strict positivity}, we employ the weak comparison principle. This principle plays a crucial role in our proofs and is also of independent interest.

\begin{theorem}[\bf Weak comparison principle]\label{thm:weak comparison principle}
  Let $u^{(1)}$ and $u^{(2)}$ be two solutions to \eqref{eq:SHE} with initial data $u^{(1)}_{0}\in L^\infty(\dR)$ and $u^{(2)}_{0}\in L^\infty(\dR)$ respectively, with $\sigma$ satisfying Assumption~\ref{assumption:basic condition for sigma}. Suppose that \eqref{eq:positivity_condition} holds and $u^{(1)}_{0}(x)\geq u^{(2)}_{0}(x)$ for all $x\in\dR$. Then we have for any $T>0$ 
  \begin{equation*}
    \dP\left(u^{(1)}(t,x)\geq u^{(2)}(t,x) \text{ for all }(t,x) \in (0,T) \times \dR  \right) =1.
  \end{equation*} In particular, the solution $u$ in Theorem \ref{thm:strict positivity} is nonnegative if $u_0 \in L^\infty(\dR)$ is nonnegative.
\end{theorem}

\begin{remark} 
  It is worth noting that Mytnik and Perkins \cite{mytnik2011pathwise} established the pathwise uniqueness when $\sigma(u)$ is H\"older continuous with an exponent greater than 3/4. However, the combination of \eqref{eq:csp_condition} and \eqref{eq:difference of sigma bounded by rho} implies that $\sigma(u)$ is indeed H\"older continuous with an exponent greater than 3/4, thus ensuring that our condition is stronger than that of Mytnik and Perkins. We point out that our method provides the weak comparison principle and strict positivity whereas they are not shown in \cite{mytnik2011pathwise}. 
\end{remark}

\begin{remark}
In Theorems \ref{thm:strict positivity} and \ref{thm:weak comparison principle}, the assumptions on the regularity of initial data can be relaxed considerably. Indeed, one can prove that initial functions are allowed to be any measure $\mu$ on $\dR$ such that $\int_{\dR} \exp(-cx^2)\mu(dx) <\infty$ for all $c>0$ following the ideas in the proof of \cite[Theorem 1.3]{chen2017comparison}.  
\end{remark}

We now introduce another  condition on $\sigma$ that ensures CSP for solutions to \eqref{eq:SHE}. It is important to note that this condition does not guarantee the uniqueness of a weak solution, and indeed, uniqueness is not required for our purposes. We will show that under this condition, any nonnegative weak solution to \eqref{eq:SHE} exhibits  CSP.

\begin{theorem}[\bf Existence and Compact support property]\label{thm:cpt_support}
Suppose $\sigma:\dR \to \dR$ satisfies Assumption \ref{assumption:basic condition for sigma} and there exists $\alpha\in(\frac{5}{2},\infty)$ and $v_2 = v_2(\alpha)>0$ such that for all $v\in(0,v_2)$
\begin{equation}\label{eq:csp_condition}
\frac{v}{\sigma(v)} \leq (-\log v)^{-\alpha}. 
\end{equation}
Let $u_0\in C^{1/2}(\dR)$ be nonnegative and compactly supported. Let $\tau$ be any bounded stopping time. Then, we have the following:
\begin{itemize}
\item[(i)] There exists a stochastically weak solution $u\in C([0,\tau]; C_{tem} (\dR))$ to \eqref{eq:SHE} that is also nonnegative. Additionally, for any weak solution $u\in C([0,\tau]; C_{tem} (\dR))$ to \eqref{eq:SHE} and any $\gamma \in (0,1/4)$, the following holds for every $a>0:$
\begin{equation}
\| \Psi_a u\|_{C^\gamma([0,\tau]\times \dR)} < \infty \quad \text{almost surely,}
\end{equation}
where $\Psi_a(x) = \Psi_a(|x|)= \frac{1}{\cosh(a|x|)}.$
\item[(ii)] For  any nonnegative weak solution $u\in C([0,\tau]; C_{tem} (\dR))$ to \eqref{eq:SHE}, we have that 
\begin{equation}\label{eq:compact support property}
\dP \Big(\text{There exists }R= R(\omega)>0 \text{ such that }  \supp(u(t,\cdot)) \subseteq [-R,R] \Big) =1,
\end{equation} where $\supp(u(t,\cdot))$ denotes the support of the solution $u(t,\cdot)$ for each $t\geq0$.
\end{itemize}
\end{theorem}

\begin{remark}
Theorem \ref{thm:strict positivity} shows that the solution is strictly positive whenever $\sigma(u) \lesssim u|\log u|^\alpha$ for all $0\leq u\ll 1$, for $\alpha \in (0, 1/4)$. On the other hand, Theorem \ref{thm:cpt_support} indicates that whenever $\sigma(u) \gtrsim u|\log u|^\alpha$ for some $\alpha>5/2$, any weak solution is compactly supported. For instance, if $\sigma(u)= u|\log u|^{\beta}|\log\log 1/u |^\gamma$ for $0\leq u \ll 1$ and any $\gamma>0$, we have a strictly positive solution if $\beta \in (0, 1/4)$. Conversely, any nonnegative weak solution is compactly supported when $\beta>5/2$. These findings extend the well-established results for $\sigma(u)=u^\gamma$ and demonstrate their consistency with Theorems \ref{thm:strict positivity} and \ref{thm:cpt_support}.
\end{remark}

\begin{remark}
For the deterministic heat equation \eqref{eq:HE}, the condition for CSP is given by \eqref{eq:condition for CSP of deterministic PDE}, which is equivalent to
\begin{equation}\label{eq:HE_cond}
\sum_{k=1}^\infty\sqrt{ \frac{e^{-k}}{\sigma(e^{-k})} } < \infty.
\end{equation}
The following heuristic argument provides an intuitive understanding of this condition: Suppose at time $t_{k-1}$, $u(t_{k-1},x) \approx 1$ for $x \in [-x_{k-1}, x_{k-1}]$. Define $t_k$ as the time at which the solution, affected only by the negative nonlinear drift  $-\sigma$ (ignoring the Laplacian), decreases to half of its previous value of $u(t_{k-1}, x_{k-1})$. Let $x_k$ be the spatial point where this decay occurs. Then, considering only the effect of $\sigma$, $t_{k+1} - t_k$ can be roughly approximated by $e^{-k}/\sigma(e^{-k})$. Due to the heat scaling, the spatial spread $x_{k+1}-x_k$ can be approximated by $\sqrt{e^{-k}/\sigma(e^{-k})}$. Thus, CSP is expected to hold if $\sum_{k=1}^\infty (x_{k+1}-x_k)$ is finite, as indicated by \eqref{eq:HE_cond}. While a rigorous proof of this idea is beyond the scope of this remark, it serves as a useful starting point for understanding conditions for CSP and raises the question of whether there exists a parallel condition for the stochastic heat equation \eqref{eq:SHE}. Although Theorems \ref{thm:strict positivity} and \ref{thm:cpt_support} do not provide a sharp condition for \eqref{eq:SHE}, by extending the heuristic argument into SPDEs, we might predict that $t_{k+1} - t_k \approx \left[ e^{-k}/\sigma(e^{-k})\right]^4$ since $u(\cdot,x)$ for each $x\in\dR$ behaves as a fractional Brownian motion with Hurst index $1/4$ (see \cite[Section 3]{khoshnevisan2014analysis}). This leads to the proposed critical condition for CSP in \eqref{eq:SHE}:
\begin{equation*}
\sum_{k=1}^\infty \left[ \frac{e^{-k}}{\sigma(e^{-k})} \right]^2 < \infty.
\end{equation*} 
\end{remark}

\subsection*{Proof ideas}

We now briefly discuss the proofs of Theorems \ref{thm:strict positivity} and \ref{thm:cpt_support}, focusing on strict positivity (Theorem \ref{thm:strict positivity} (ii)) and CSP (Theorem \ref{thm:cpt_support} (ii)). For the proof of Theorem \ref{thm:strict positivity} (ii), we adopt the strategy introduced by Mueller \cite{mueller1991support}, which involves a mild solution (see Definition \ref{def:definition of mild solution}) consisting of two components: the solution to the deterministic heat equation, which is inherently strictly positive, and the stochastic integral part influenced by noise. Mueller's approach relies on the observation that if the noise-induced fluctuations are sufficiently small over a short duration, the solution's strict positivity is maintained. In our analysis, we focus on tightly controlling the influence of noise to ensure strict positivity, which is an improvement over Mueller's original methodology in that we deal with a more general class of nonlinearities $\sigma$ (see Proposition~\ref{prop:key prop for strict positivity}).

On the other hand, for the proof of Theorem \ref{thm:cpt_support} (ii), we follow the strategy by Krylov \cite{krylov1997result}, which introduces a framework based on weak solutions (see Definition \ref{def:definition of weak solution}) to establish CSP for $\sigma(u)=u^\gamma$. Central to this framework is Lemma 2.1 from \cite{krylov1997result}, an analogue of which appears as Lemma \ref{lem:key estimate for CSP} in this paper. Since Krylov's work focused on $\sigma(u)=u^\gamma$, he was able to use Jensen's inequality to obtain Lemma 2.1 in \cite{krylov1997result}. However, given the broader scope of our $\sigma$ beyond $u^\gamma$, applying Jensen's inequality directly proves challenging. To overcome this challenge, our method introduces a novel partitioning of the time interval according to the size of the solution (see Section \ref{subsec:proof of props}). We examine the behavior of the solution in distinct time intervals: those where the solution exhibits relatively low values and those where it is comparably higher. This approach allows us to formulate Lemma \ref{lem:key estimate for CSP} effectively, adapting Krylov's framework to our more general setting.

We conclude the Introduction with an outline of the paper and some notation. Section \ref{sec:preliminaries} provides definitions of mild and weak solutions of \eqref{eq:SHE}. The proofs of Theorems \ref{thm:strict positivity} and \ref{thm:weak comparison principle} are presented in Sections \ref{sec:weak comparison} and \ref{sec:strict positivity}. Finally, Section \ref{sec:proof of compact support property} is dedicated to the proof of Theorem \ref{thm:cpt_support}.

\subsection*{Notation}

We define $C_{tem}(A)$ and $C^\gamma(A)$ for $A\subset \dR$ by $$ C_{tem}(A) := \left\{ u\in C(A):\sup_{x\in A}|u(x)|e^{-a |x|}<\infty\text{ for any } a>0 \right\}, 
$$ where $C(A)$ is the space of continuous functions on $A$ and 
$$
C^\gamma(A) := \left\{ u\in C(A)\,;\:\, \|u\|_{C^\gamma(A)}:=\sup_{x\in A}|u(x)| + \sup_{x\neq y\in A}\frac{|u(x)-u(y)|}{|x-y|^\gamma} <\infty  \right\}.
$$ For simplicity, we may write $C^\gamma$ and $L^p$ for $C^\gamma(\dR)$ and $L^p(\dR)$ respectively. For any subsets $A$ and $B$ of Banach spaces, We denote by $C(A;B)$  the space of continuous functions from $A$ to $B$. We define $\cS := \cS(\dR_+\times \dR)$ by
\[
  \cS(\dR_+\times \dR):= \left\{ u\in C^\infty(\dR_+\times \dR)\,: \, \sup_{x\in\dR_+\times\dR} |x^\alpha (D^\beta u(x))|<\infty \text{ for all } (\alpha,\beta) \in \dN^2 \right\}.
\] Throughout this paper, we write $C=C(a_1, a_2,...,a_k)$ and $C_i=C_i(a_1, a_2,...,a_k)$ for $i,k\in \dN$ as  generic constants which depend only on $a_1, a_2,...,a_k$. The nuemeric value of $C$ can be changed line by line. We write $a\lesssim b$ when $a\leq C\cdot  b$ holds for a universal constant $C>0$. For $a,b\in\dR$, we write $a\wedge b$ and $a\vee b$ for $\min(a,b) $ and $\max(a,b)$ respectively.

\section{Preliminaries}\label{sec:preliminaries}

We present some preliminaries for the SHE \eqref{eq:SHE}. Let $(\Omega , \cF,  \dP) $ be a complete probability space. We denote by $\{ \dot{W}(t,x) : (t,x) \in \dR_+ \times \dR\}$ the space-time white noise. More precisely, $\{W(\phi) : \phi \in\cS(\dR_+\times \dR)\}$ is a mean-zero Gaussian process on $(\Omega , \cF, \dP)$ with the covariance functional 
\begin{equation}
\label{def of covaiance functional}
\begin{aligned}
\dE \left[W(\phi)W(\psi)\right] 
&=\int_0^\infty\int_{\dR} \phi(t,x)\psi(t,x) dx dt \quad \text{for }\phi,\psi\in \cS.
\end{aligned}
\end{equation} Here, we write $W(\phi)$ for
\begin{equation*}
  W(\phi) : = \int_0^\infty\int_{\dR} \phi(t,x) W(dtdx).
 \end{equation*} We now regard $\{ W(t,x) : (t,x) \in \dR_+ \times \dR\}$ as the martingale measure (see \cite{walsh1986introduction,dalang1999extending}) extended from $\{W(\phi) : \phi \in\cS(\dR_+\times \dR)\}$.  Let $\{ \cF_t : t\geq 0\} $ be the filtration of the $\sigma-$fields generated by $\dot{W}$ satisfying the usual conditions, and $\cP$ be the predictable $\sigma-$field related to $\cF_t$. We also set $\cF: = \vee_{t\geq 0 }\cF_t$. We define $\|X \|_{p}$ for a random variable $X:\Omega \to \dR$ and $p>0$ by 
 \[ \|X \|_{p} = (\dE\left[|X|^p \right])^{1/p}. \]
We note that in Theorem \ref{thm:strict positivity}, we focus on the concept of a \emph{mild solution} to  \eqref{eq:SHE}, while in Theorem \ref{thm:cpt_support}, we consider a \emph{weak solution} to \eqref{eq:SHE}. In general, particularly when the function $\sigma$ is well-behaved, mild and weak solutions are shown to be equivalent (see \cite{da2014stochastic,khoshnevisan2014analysis}). For proving strict positivity (Theorem \ref{thm:strict positivity} (ii)), we choose to work with mild solutions because they are more suitable for our approach. On the other hand, our strategy for demonstrating CSP (Theorem \ref{thm:cpt_support} (ii)) is better suited to weak solutions. We will provide definitions of both mild and weak solutions in the following:
 
\begin{definition}\label{def:definition of mild solution} 
  We say that $u$ is a mild solution of \eqref{eq:SHE} if for all $(t,x)\in\dR_+\times \dR$, 
  \begin{equation}\label{eq:definition of mild solution}
    u(t,x) = (G(t,\cdot) *u_0)(x) + \int_0^t \int_{\dR} G(t-s,x-y)\sigma(u(s,y))W(dsdy),
  \end{equation} almost surely, and the following properties hold:
  \begin{itemize}
    \item [(1)] $u(t,x)$ is $\cF_t$-measurable for all $(t,x)\in\dR_+ \times \dR$;
    \item [(2)] $u$ is jointly measurable with respect to $\mathcal{B}(\dR_+\times\dR)\times \cF $, where $\mathcal{B}(\dR_+\times\dR)$ denotes the Borel $\sigma-$field of $\dR_+\times\dR$ ;
    \item [(3)] $\dE[|u(t,x) |^2]<+\infty$ for all $(t,x)\in \dR_+\times \dR$;
    \item [(4)] The function $(t,x) \mapsto u(t,x)$ mapping $\dR_+\times \dR$ into $L^2(\Omega)$ is continuous.
  \end{itemize}
\end{definition} Note that the stochastic integral \eqref{eq:SHE} can be  understood in the sense of Walsh (see \cite{walsh1986introduction,dalang1999extending}) and $G(t,x)$ is the heat kernel on $\dR$, which is given by 
\[  G(t,x) : = \frac{1}{\sqrt{2\pi t}} \exp\left( -\frac{x^2}{2t} \right).\] 

In the sequel, $\tau$ is any given bounded stopping time, i.e., $\tau\leq  T$ for some non-random number $T>0$. 

\begin{definition} 
\label{def:definition of weak solution} We say that $u$ is a weak solution of \eqref{eq:SHE} on $[0, \tau]$ if $u$ is a $C_{tem}(\dR)$-valued function $u = u(t,\cdot)$ defined on $\Omega\times[0,\tau]$ satisfying that if for any $\phi\in \cS$  
\begin{itemize}
\item[(1)]
The process $\int_{\dR} \phi(x)u(t\wedge \tau,x) dx$ is well-defined, 
$\cF_t$-adapted, and continuous;

\item[(2)]
The process 
$$ \int_{\dR} \left| \sigma(u(t \wedge \tau,x))\phi(x) \right|^2 dx $$ 
is well-defined, $\cF_t$-adapted, measurable with respect to $(\omega,t)$, and 
\begin{equation}
\label{eq:quadratic variation of stochastic part}
\begin{gathered}
 \int_0^\tau \int_{\dR}\left| \sigma(u(t,x))\phi(x) \right|^2 dx  dt <\infty
\end{gathered}
\end{equation}
almost surely;

\item[(3)]
The equation
\begin{equation} \label{eq:sol_int_eq_form}
\begin{aligned}
  (u&(t,\cdot),\phi)_{L^2(\dR)}\\ &= (u_0,\phi)_{L^2(\dR)}+\int_0^t(u(s,\cdot), (2^{-1}\partial_x^2\phi)(s,\cdot))_{L^2(\dR)} ds + \int_0^t\int_{\dR}\sigma(u(s,x))\phi(x) W(dsdx)
\end{aligned}
\end{equation} holds for all $t\leq \tau$ almost surely.

\end{itemize}

\end{definition}

\section{Well-posedness  and weak comparison principle}\label{sec:weak comparison}
Throughout this section, we assume that $\sigma$ satisfies Assumption \ref{assumption:basic condition for sigma}. Our goals in this section are to demonstrate the existence and uniqueness of a mild solution, as outlined in Theorem \ref{thm:strict positivity} (i), subject to the condition \eqref{eq:positivity_condition}, and to establish the weak comparison principle presented in Theorem \ref{thm:weak comparison principle}, which is an essential step towards proving strict positivity (see Theorem \ref{thm:strict positivity} (ii)). To lay the groundwork for these proofs, we begin by examining the pseudo-Lipschitz continuity of the function $\sigma$ in Assumption \ref{assumption:basic condition for sigma}, a property that will be frequently used throughout this section.

\begin{lemma}\label{lem:property of sigma}
  For any function $\sigma$ satisfying Assumption \ref{assumption:basic condition for sigma} with the constants $\fd>0$ and $L_\fd>0 $ therein, we have the following property: For any $\delta \in (0, \fd) $, we have 
  \begin{equation}\label{eq:delta-Lipschitz continuity of sigma}
    |\sigma(u) - \sigma(v) | \leq \left( L_\fd \vee \delta^{-1}\sigma(\delta) \right)|u-v| + \sigma (\delta), \quad \text{for all }u,v\in \dR.
  \end{equation} 
  Moreover, we have 
  \begin{equation}\label{eq:difference of sigma bounded by rho}
    |\sigma(u) - \sigma(v) | \leq   L_\fd |u-v| + \sigma(|u-v|), \quad \text{for all }u,v\in \dR,
  \end{equation}
\end{lemma}
\begin{proof}
  We prove the lemma for the case $u\geq v \geq 0 $ since $\sigma(u)\1_{u\leq 0 } = 0 $ (see Assumption \ref{assumption:basic condition for sigma}). If $u\geq v \geq \fd > \delta$, then (b) of Assumption \ref{assumption:basic condition for sigma} immediately shows that \eqref{eq:delta-Lipschitz continuity of sigma} holds. For the case $u\geq \fd \geq v \geq \delta$, we have
  \begin{equation*}
  \begin{aligned}
        |\sigma(u)-\sigma(v)| &\leq \sigma(u) - \sigma(\fd) + \sigma(\fd) -\sigma (v) \\
        &\leq L_\fd(u-\fd) + \frac{\sigma(\fd)}{\fd} \fd  -  \frac{\sigma(v)}{v} v\\
        &\leq L_\fd(u-\fd) + \frac{\sigma(v)}{v}(d-v)\\
        &\leq \left( L_\fd \vee \frac{\sigma(\delta)}{\delta}\right) (u-v).
  \end{aligned}
   \end{equation*}
In the second line, we have used Assumption \ref{assumption:basic condition for sigma} (b) for the first term. In the third line, we used Assumption \ref{assumption:basic condition for sigma} (a) for the second term. In the last line, we again used Assumption \ref{assumption:basic condition for sigma} (a). For the case $u\geq \fd \geq \delta \geq v$, we similarly obtain that
\begin{equation*}
  \begin{aligned}
        \sigma(u)-\sigma(v) &\leq \sigma(u) - \sigma(\fd) + \sigma(\fd) -\sigma(\delta ) +\sigma(\delta ) -\sigma (v) \\
        &\leq L_\fd (u-\fd) +  \frac{\sigma(\delta)}{\delta}( \fd - \delta)+ \frac{\sigma(\delta)}{\delta }(\delta - v) \\
        &\leq \left( L_\fd \vee \frac{\sigma(\delta)}{\delta}\right) (u-v).
  \end{aligned}
   \end{equation*} 
For the last term in the second line, we have used   
   \begin{equation*}
    \sigma(\delta)- \sigma(v)=\frac{\sigma(\delta)}{\delta }\delta - \frac{\sigma(v)}{v}v \leq \frac{\sigma(\delta)}{\delta }(\delta - v),
   \end{equation*}
due to the monotonicity of $v\mapsto \sigma(v)/v$ in Assumption \ref{assumption:basic condition for sigma} (a). The case $\fd \geq u\geq v \geq \delta$ can be proven similarly, as
\begin{equation*}
\sigma(u)-\sigma(v) = \frac{\sigma(u)}{u}u - \frac{\sigma(v)}{v}v \leq \frac{\sigma(v)}{v}(u-v)\leq \frac{\sigma(\delta)}{\delta}(u-v),
\end{equation*}
since $\sigma(u)/u\leq \sigma(v)/v\leq \sigma(\delta)/\delta$. For the case $\fd \geq u\geq \delta \geq v$, we observe that
\begin{equation*}
\sigma(u)-\sigma(v) \leq \frac{\sigma(u)}{u}u -\frac{\sigma(\delta)}{\delta}\delta+\frac{\sigma(\delta)}{\delta}\delta - \frac{\sigma(v)}{v}v \leq\frac{\sigma(\delta)}{\delta}(u-v),
\end{equation*}
again owing to the monotonicity of $v\mapsto \sigma(v)/v$. In the final case $\delta \geq u\geq v$, we can simply bound
\begin{equation*}
\sigma(u) - \sigma(v) \leq \sigma( u ) \leq \sigma(\delta).
\end{equation*}
To prove \eqref{eq:difference of sigma bounded by rho}, observe that for all $\fd\geq u \geq v$
\begin{equation*}
\frac{(u-v)\sigma(u)}{u}\leq \sigma(u-v), \quad \frac{v\sigma(u)}{u}\leq \sigma(v),
\end{equation*}
which implies that
\[ \sigma(u) -\sigma(v) \leq \sigma(u-v).\]
If $u\geq \fd \geq v$, then we have
\begin{equation*}
\begin{aligned}
\sigma(u) - \sigma(v) &\leq \sigma (u) - \sigma( \fd ) + \sigma(\fd )- \sigma(v)\
\leq L_\fd (u-\fd) + \sigma( \fd - v)\
 \leq L_\fd (u-v) + \sigma(u-v)
\end{aligned}
\end{equation*}
by Assumption \ref{assumption:basic condition for sigma}. The case $u\geq v \geq \fd$ is proven immediately by Assumption \ref{assumption:basic condition for sigma} (b).
\end{proof}

The existence and uniqueness of a mild solution to \eqref{eq:SHE} can be established using a classical fixed point argument or a Picard iteration method when $\sigma$ is globally Lipschitz continuous. However, in our case, the fact that $\sigma(u)/u$ can blow up as $u\to 0$ makes these classical methods inapplicable. Lemma \ref{lem:property of sigma} suggests a potential workaround: by truncating $\sigma(u)$ when $u$ is small, we can make the truncated $\sigma$ globally Lipschitz. More specifically, we consider the following stochastic heat equations: For each integer $m\geq 1$, we consider the following stochastic heat equation: 
\begin{equation}\label{eq:M-truncated SPDE}
  \begin{cases}
      \partial_tu^{(m)}(t,x) = \frac{1}{2}\partial^2_x u^{(m)}(t,x) + \sigma^{(m)}(u^{(m)}(t,x))\dot{W}(t,x), \quad(t,x)\in \dR_+ \times \dR,\\
      u^{(m)}(0,x) = u_0(x), \quad x\in\dR,
    \end{cases}
\end{equation} where 
\begin{equation}\label{eq:definition of sigma^m}
  \sigma^{(m)}(u) : = \sigma (u)\1_{u\geq 1/ m } + m \sigma(m^{-1}) u\1_{0\leq u< 1/m } \quad \text{for all }u\geq 0,
\end{equation} and $\sigma^{(m)}(u)\1_{u\geq 0} =0$. We can follow the same line of the proof of Lemma \ref{lem:property of sigma} to observe that there exists $m_0\geq 1 $ such that for all $m\geq m_0 $, 
\begin{equation*}
  |\sigma^{(m)}(u) -\sigma^{(m)}(v) | \leq (m \sigma(m^{-1})  \vee L_\fd)|u-v|  \quad \text{for all }u,v\in\dR .
\end{equation*} 
Note that the only difference from \eqref{eq:delta-Lipschitz continuity of sigma} occurs when $0\leq u,v\leq 1/m$. However, in this case, we get $|\sigma^{(m)}(u) -\sigma^{(m)}(v) | \leq m \sigma(m^{-1})|u-v|$ from the definition of $\sigma^{(m)}$ in \eqref{eq:definition of sigma^m}. Without loss of generality, throughout the remainder of this section, we assume that for all $m\geq m_0$, 
\begin{equation}\label{eq:lipschitz continuity of sigma^m}
  |\sigma^{(m)}(u) -\sigma^{(m)}(v) | \leq m \sigma(m^{-1})|u-v|  \quad \text{for all }u,v\in\dR.
\end{equation} Indeed, if we assume that $m\sigma(m^{-1})\leq L_\fd$ for all $m\geq 1$, then $\sigma^{(m)}$ is globally Lipschitz continuous for all $m\geq 1$, which would simplify the proof considerably. We also note that \eqref{eq:lipschitz continuity of sigma^m} implies that $\sigma^{(m)}$ is globally Lipschitz continuous for each $m\geq m_0$. Therefore, for each $m\geq m_0$, there exists a unique mild solution to \eqref{eq:M-truncated SPDE}, which is also H\"older continuous in time and space. Moreover, it is known that the weak comparison principle holds for such equations, as demonstrated in the works of  Shiga \cite{shiga1994two} and  Chen and Kim \cite{chen2017comparison}. The weak comparison principle is crucial for establishing the nonnegativity of solutions and comparing solutions with different initial data. The following lemma, which can be found in \cite[Corollary 2.4]{shiga1994two} and \cite[Theorem 1.1]{chen2017comparison}, states the weak comparison principle for the solutions of \eqref{eq:M-truncated SPDE}:

\begin{lemma}\label{lem:weak comparison for u^m}
  Let $u^{(m,1)}$ and $u^{(m,2)}$ be two mild solutions to \eqref{eq:M-truncated SPDE} with initial data $u^{(1)}_0\in L^\infty(\dR)$ and $u^{(2)}_0\in L^\infty(\dR)$ respectively. Suppose that $u^{(1)}_0(x) \geq u^{(2)}_0(x)$ for all $x\in\dR$. Then, we have 
    \begin{equation*}
    \dP\left(u^{(m,1)}(t,x)\geq u^{(m,2)}(t,x) \text{ for all }(t,x) \in \dR_+\times \dR  \right) =1.
  \end{equation*} In particular, the solution $u^{(m)}$ of \eqref{eq:M-truncated SPDE} is nonnegative if $u_0 \in L^\infty(\dR)$ is nonnegative.
\end{lemma}
To prove the existence and uniqueness of a mild solution to \eqref{eq:SHE}, we first establish moment estimates for the solutions $u^{(m)}$ to the truncated equation \eqref{eq:M-truncated SPDE}. These estimates play a key role in showing that the sequence $u^{(m)}$ converges as $m \to \infty$. From \eqref{eq:lipschitz continuity of sigma^m}, we have the following  moment estimate for the solution $u^{(m)}$ of \eqref{eq:M-truncated SPDE}: For all $p\geq 2 $ there exists $C=C(\|u_0\|_{L^\infty} , L_\fd)>0$  such that for any $T>0$ 
\begin{equation}\label{eq:pth moment bound for lipschitz sigma}
  \sup_{0\leq t\leq T} \sup_{x\in\dR}\| u^{(m)}(t,x)\|_p^p\leq C^p\exp \left( Cm^4 \sigma^4(m^{-1}) p^3 T\right )
  \end{equation} (see \cite{khoshnevisan2014analysis,chen2017comparison,le2023superlinear} for instance). Observe that the RHS diverges as $m\to \infty$. However, from the definition of $\sigma^{(m)}$ in \ref{eq:definition of sigma^m} and Assumption \ref{assumption:basic condition for sigma}, there exists $C=C(L_\fd)>0$ such that uniformly over all $m\geq 1 $, for all $u\in\dR$
  \begin{equation} \label{eq:uniform bound for sigma m}
    |\sigma^{(m)}(u)| \leq C(1+|u|).
  \end{equation} From this fact, it is quite standard to deduce that the moments of $u^{(m)}$ can be bounded above uniformly in $m\geq 1 $. We present the proof for the convenience of the reader. 
  \begin{lemma}\label{lem:uniform moment bound for u^m} Let $u^{(m)}$ be a mild solution to \eqref{eq:M-truncated SPDE} with initial data $u_0\in L^{\infty}(\dR)$. There exists $C=C(L_{\fd}, \|u_0\|_{L^\infty})>0$ such that for all $p\geq 2$, $m\geq 1 $ and $T>0$,
    \begin{equation*}
      \sup_{0\leq t\leq T}\sup_{x\in \dR} \| u^{(m)}(t,x) \|^p_p \leq C^p \exp (C p^3 T). 
    \end{equation*}
  \end{lemma} 
\begin{proof} Let $p\geq 2 $ and $T>0$. By the Burkholder-Davis-Gundy (BDG) inequality (see \cite[Theorem B.1]{khoshnevisan2014analysis}), there exists $C=C(L_\fd)>0$ such that for all sufficiently large $m\geq 1 $,
\begin{align*}
  \|u^{(m)}(t,x) \|_p^2 &\leq C\|u_0\|_{L^\infty}^2 +Cp\int_0^t \int_{\dR} G^2(t-s,y-x)\left(1+  \| u^{(m)}(s,y)\|_p^2 \right) dyds\\
    &\leq  C\|u_0\|_{L^\infty}^2 + Cp\int_0^t (t-s)^{-1/2}  \sup_{x\in\dR}\| u^{(m)}(s,x)\|_p^2ds,
\end{align*} where in the second inequality we used that for all $x\in \dR$
\begin{equation}\label{eq:heat kernel bound}
  \int_\dR G^2(t-s,y-x) dy  \lesssim (t-s)^{-1/2}.
\end{equation} Then, we can apply a version of Gronwall's inequality (see \cite[Corollary 2]{ye2007generalized}) to obtian that 
\begin{equation*}
  \|u^{(m)}(t,x) \|_p^2  \leq C\|u_0\|_{L^\infty} E_{\frac{1}{2}}\left( C p t^{\frac{1}{2}}\right),
\end{equation*} where $E_{\frac{1}{2}}$ denotes the Mittag-Leffler function with index $\frac{1}{2}$. Since $E_{\frac{1}{2}}(z) \lesssim \exp(z^2)$ for $z\in \dR$, we have that 
\begin{equation*}
  \|u^{(m)}(t,x) \|_p^2  \leq C\|u_0\|_{L^\infty} \exp\left( C^2 p^2 t\right).
\end{equation*} This completes the proof.
\end{proof}

Next, we aim to show that the sequence of solutions $u^{(m)}$ to the truncated equation \eqref{eq:M-truncated SPDE} converges as $m \to \infty$. We also need to verify that the limit function is indeed a mild solution to the original stochastic heat equation \eqref{eq:SHE}. This step is crucial for proving the existence and uniqueness of a mild solution, as the Lipschitz coefficient of $\sigma^{(m)}$ in \eqref{eq:lipschitz continuity of sigma^m} may blow up as $m \to \infty$, making the passage to the limit a non-trivial task. However, the assumption \eqref{eq:positivity_condition} allows us to control the growth of the nonlinear term, which leads to the convergence of the sequence $u^{(m)}$ and the uniqueness of a mild solution to \eqref{eq:SHE}.

\begin{proposition}\label{prop:existence of strong mild solution} Let $u^{(m)}$ be a mild solution to \eqref{eq:M-truncated SPDE} with initial data $u_0\in L^{\infty}(\dR)$. Suppose that $\sigma:\dR\to\dR$ satisfies \eqref{eq:positivity_condition}.  Then, for any $T>0$, there exists a unique mild solution $u\in C([0, T]\times \dR ; L^2(\Omega))$ to \eqref{eq:SHE} such that 
  \begin{equation*}
    \lim_{m\to \infty} \sup_{t\leq T}\sup_{x\in \dR} \|u^{(m)}(t,x) -u(t,x) \|_2=0.
  \end{equation*} 
\end{proposition}

\begin{proof} 
Throughout the proof, we fix a constant $T>0$ and assume that $t\in [0,T]$. We first  show that for each $(t,x) \in [0,T]\times \dR$,  $\{u^{(m)}(t,x)\}_{m\geq 1} $ is a Cauchy sequence in $L^2(\Omega).$ Set 
\begin{equation*}
   \bJ_m(t,x) : = u^{(m)}(t,x) -u^{(m-1)}(t,x).
\end{equation*} By the It\^o-Walsh isometry, we have 
\begin{equation}\label{eq:eq for bold J_m}
  \begin{aligned}
    \|\bJ_m(t,x) \|_2^2 = \int_0^t \int_{\dR } G^2(t-s,y-x) \| \sigma^{(m)}(u^{(m)}(s,y) )-\sigma^{(m-1)}(u^{(m-1)}(s,y)) \|_2^2 dyds.
  \end{aligned}
\end{equation} Observe that
\begin{equation*}
  \begin{aligned}
    \| \sigma^{(m)}(u^{(m)}(s,y) )-\sigma^{(m-1)}(u^{(m-1)}(s,y)) \|_2^2 \lesssim &\| \sigma^{(m)}(u^{(m)}(s,y) )-\sigma^{(m)}(u^{(m-1)}(s,y)) \|_2^2\\
    &+ \| \sigma^{(m)}(u^{(m-1)}(s,y) )-\sigma^{(m-1)}(u^{(m-1)}(s,y)) \|_2^2.
  \end{aligned}
\end{equation*} From the definition of $\sigma^{(m)}$ (see \eqref{eq:definition of sigma^m}), we can follow the same line in the proof of Lemma \ref{lem:property of sigma} to obtain that for all sufficiently large $m\geq1 $
\begin{equation*}
  \| \sigma^{(m)}(u^{(m)}(s,y) )-\sigma^{(m)}(u^{(m-1)}(s,y)) \|_2^2 \lesssim m^2\sigma^2(m^{-1})\|\bJ_m(s,y) \|_2^2. 
\end{equation*} Moreover, by \eqref{eq:definition of sigma^m} and Assumption \ref{assumption:basic condition for sigma}, there exists $m_0\geq 1 $ such that for all $m\geq m_0$
\begin{align*}
  &|\sigma^{(m)}(u )-\sigma^{(m-1)}(u) | \\
  &\leq | \sigma(u) - (m-1)\sigma((m-1)^{-1})u|\cdot\1_{m^{-1}\leq u < (m-1)^{-1}} \\&\quad\quad+ |m\sigma(m^{-1}) - (m-1)\sigma((m-1)^{-1})| u \cdot\1_{0\leq u< m^{-1}}\\
  &\leq | \sigma((m-1)^{-1}) - (m-1)\sigma((m-1)^{-1})m^{-1}|+ |m\sigma((m-1)^{-1}) - (m-1)\sigma((m-1)^{-1})| \cdot m^{-1}\\
  &\leq 2 m^{-1} \sigma((m-1)^{-1}).
\end{align*} Therefore, it turns out that 
\begin{equation}\label{eq:difference of sigma_m and sigma_m-1}
  \| \sigma^{(m)}(u^{(m)}(s,y) )-\sigma^{(m-1)}(u^{(m-1)}(s,y)) \|_2^2 \lesssim m^2\sigma^2(m^{-1})\|\bJ_m(s,y) \|_2^2 + m^{-2}\sigma^2((m-1)^{-1}).
\end{equation} Let $\fJ_m(t) : = \sup_{ x\in \dR} \|\bJ_m(t,x)\|_2^2 $. Then, from \eqref{eq:eq for bold J_m} and \eqref{eq:difference of sigma_m and sigma_m-1}, there exists $C>0$ such that 
\begin{equation*}
  \fJ_m(t) \leq C \left( m^2\sigma^2(m^{-1})\int_0^t (t-s)^{-1/2} \fJ_m(s) ds +  m^{-2}\sigma^2((m-1)^{-1})\right),
\end{equation*} where we have used \eqref{eq:heat kernel bound}. Then, we use the Gr\"onwall inequality as in the proof of Lemma \ref{lem:uniform moment bound for u^m} to obtian that for some $C>0$,
\begin{equation*}
  \fJ_m(t) \leq C m^{-2}\sigma^2((m-1)^{-1}) E_{\frac{1}{2}}\left( C m^2\sigma^2(m^{-1}) t^{\frac{1}{2}} \right).
\end{equation*} Again using the fact $E_{\frac{1}{2}}(z) \leq \exp(z^2) $ for all $z\in \dR$, we have 
\begin{equation}\label{eq:Grownall's inequality for sigma_m}
    \fJ_m(t) \leq C m^{-2}\sigma^2((m-1)^{-1}) \exp\left( C^2 m^4\sigma^4(m^{-1}) T \right),
\end{equation} for all $t\in [0,T]$.
To conclude that $u^{(m)}(t,x)$ is Cauchy, we will show that $\sum_{m=1}^{\infty} (\fJ_m(t))^{1/2} <\infty$. Specifically, for all $n>m\geq 1$, we have
\begin{equation*}\label{eq:difference u^n-u^m}
\| u^{(n)}(t,x) - u^{(m)}(t,x) \|_2 \leq \sum_{k=m+1}^{n} \left(\fJ_k(t) \right)^{1/2}\quad \text{for all }t\geq 0.
\end{equation*}
Now it is clear that to show $\sum_{m=1}^{\infty} (\fJ_m(t))^{1/2} <\infty$ is sufficient for the proof. To prove this, we claim that for any constant $c>0$, there exists $m_1\geq 1$ such that for all $m\geq m_1$
\begin{equation}\label{eq:claim to end Gronwall inequality for mild solution}
m^{-1}\sigma((m-1)^{-1}) \exp ( cm^4\sigma^{4}(m^{-1}) )\leq m^{-3/2}.
\end{equation}
This inequality follows from the assumption \eqref{eq:positivity_condition}. In particular, by \eqref{eq:positivity_condition}, there exist constants $\alpha\in(0,\frac{1}{4})$ and $m_1=m_1(\alpha)>0$ such that for all $m\geq m_1$
\begin{equation}\label{eq:bound for the growh of sigma}
  m \cdot \sigma(m^{-1}) \leq (\log m )^\alpha.
\end{equation} Now we find a constant $m_2= m_2(\alpha)>m_1$ such that for all $m\geq m_2$
\begin{equation*}
\begin{aligned}
    m^{-1}\sigma((m-1)^{-1}) \exp ( cm^4\sigma^{4}(m^{-1}) ) &\leq m^{-2} \log(m)^{\alpha}\cdot \exp (c (m\sigma(m^{-1}))^4)\\
    & \leq  m^{-2} (\log m)^{\alpha} \cdot \exp \left(c (\log m)^{4\alpha} \right) \\
    &\leq m^{-3/2}.
\end{aligned}
\end{equation*} This completes the proof of the claim. Now, from \eqref{eq:Grownall's inequality for sigma_m}, for any $t\in(0,T]$ there exists $C=C(T)>0$ such that 
\begin{equation}\label{eq:bound for J}
   (\fJ_m(t))^{1/2} \leq   \frac{C}{m^{3/2}},
\end{equation} which shows that as $m,n\to \infty$
\begin{equation}\label{eq:Cauchy bound for u^m}
  \sup_{0\leq t \leq T }\sup_{x\in\dR }\| u^{(n)}(t,x) - u^{(m)}(t,x) \|_2 \leq \sum_{k=m+1}^{n} \left(\fJ_k(t) \right)^{1/2} \leq \sum_{k=m+1}^{n} \frac{C}{k^{3/2}} \to 0.
\end{equation} This concludes that $\{u^{(m)}\}_{m=1}^\infty$ is a Cauchy sequence in $C([0,T]\times \dR ; L^2(\Omega)) $. Therefore, for each $(t,x) \in [0,T] \times \dR$, we can define the $L^2-$limit of $u^{(m)}(t,x)$, say $\tilde{u}(t,x)$, which also satisfies  $\tilde{u}\in C([0,T]\times \dR ; L^2(\Omega)) $. One can easily verify that $u$ is indeed a mild  solution to \eqref{eq:SHE}, which is adapted to the filtration $\{\cF_t\}_t$. Indeed,  $u^{(m)} - u^{(n)}$ is equal to ${\bf{N}}^{(m)} - {\bf{N}}^{(n)}$ where ${\bf{N}}^{(m)}$ is defined as 
\begin{equation*}
  {\bf{N}}^{(m)}(t,x) = \int_0^t \int_{\dR} G(t-s,y-x)\sigma^{(m)}(u^{(m)}(s,y))W(dsdy).
\end{equation*} This shows that we can define ${\bf{N}}$ by the limit of ${\bf{N}}^{(m)}$ in $C([0,T]\times \dR ; L^2(\Omega))$, which implies that for all $(t,x)\in[0,T]\times \dR$
\begin{equation*}
  \tilde{u}(t,x) : = (G(t,\cdot ) *u_0)(x) + {\bf N }(t,x) = \lim_{m\to \infty} u^{(m)}(t,x)
\end{equation*} is well-defined. Moreover, $\tilde{u}(t,x)$ is adapted to $\cF_t$ as a stochastic process in $C([0,T]\times \dR ; L^2(\Omega))$. We then can define the It\^o-Walsh integral
\begin{equation*}
\int_0^t \int_{\dR} G(t-s,y-x)\sigma(\tilde{u}(s,y))W(dsdy).
\end{equation*}
To justify this, we note that $\tilde{u}$ satisfies the uniform boundedness of the second moment of $\sigma(\tilde{u})$. Specifically, by \eqref{eq:uniform bound for sigma m}, there exists a constant $C=C(L_\fd, T)>0$ such that
\begin{equation*}
\sup_{0\leq t\leq T}\sup_{x\in \dR}\dE[\sigma^2(\tilde{u}(t,x))] \leq C\left(1+ \sup_{0\leq t\leq T}\sup_{x\in \dR}\dE[|\tilde{u}(t,x)|^2]\right) <\infty.
\end{equation*}
This uniform boundedness ensures the existence of the It\^o-Walsh integral involving $\sigma(\tilde{u})$.
To conclude that $\tilde{u}$ is indeed a mild solution $u$ of \eqref{eq:SHE} (see Definition \ref{def:definition of mild solution}), it remains to show that the above stochastic integral coincides with ${\bf N}(t,x)$ for each $(t,x)\in[0,T]\times \dR$. Once again by \eqref{eq:definition of sigma^m}, we have that 
\begin{equation}\label{eq:difference of sigma and sigma^m}
\begin{aligned}
  |\sigma(\tilde{u}) - \sigma^{(m)}(u^{(m)})| &\leq  |\sigma(\tilde{u}) - \sigma^{(m)}(\tilde{u})| + |\sigma^{(m)}(\tilde{u}) - \sigma^{(m)}(u^{(m)})| \\
  &\leq \sigma(m^{-1}) + m\sigma(m^{-1}) | \tilde{u} - u^{(m)}|.
\end{aligned}
\end{equation} This implies that by the It\^o-Walsh isometry,
\begin{align*}
  \dE&\left[\left|\int_0^t \int_{\dR} G(t-s,y-x)\sigma(\tilde{u}(s,y))W(dsdy) - {\bf N}^{(m)}(t,x) \right|^2\right] \\
  &= \int_0^t \int_{\dR}G^2(t-s,y-x) \| \sigma(\tilde{u}(s,y)) - \sigma^{(m)}(u^{(m)}(s,y)) \|_2^2dy ds.
\end{align*} On the other hand, by \eqref{eq:bound for the growh of sigma} and \eqref{eq:Cauchy bound for u^m}, there exists $C=C(T)>0$ such that for all $(s,y)\in [0,T] \times \dR$
\begin{equation*}
\begin{aligned}
    m\sigma(m^{-1}) \| \tilde{u}(s,y) - u^{(m)}(s,y)\|_2  &= m\sigma(m^{-1} ) \lim_{n\to \infty}\| u^{(n)}(s,y) - u^{(m)}(s,y)\|_2 \\&\leq C  (\log m)^{\alpha} \sum_{k=m+1}^\infty \frac{1}{k^{3/2}},
\end{aligned}
\end{equation*} for all sufficiently large $m\geq 1 $. Therefore, by the dominated convergence theorem, we have 
\begin{equation*}
  \int_0^t \int_{\dR} G(t-s,y-x)\sigma(\tilde{u}(s,y))W(dsdy) = \lim_{m\to \infty} {\bf N}(t,x),
\end{equation*}  which shows that $u:= \tilde{u}$ is indeed a mild solution of \eqref{eq:SHE}. 

The uniqueness of the solution follows in a similar way. Suppose that $u$ and $v$ are solutions to \eqref{eq:SHE} with same initial data. Then, by the It\^o-Walsh isometry and Lemma \ref{lem:property of sigma},
\begin{equation*}
  \begin{aligned}
    \| u (t,x) - v(t,x) \|_2^2 &= \int_0^t \int_{\dR} G^2(t-s,y-x) \| \sigma(u(s,y)) -\sigma(v(s,y)) \|_2^2 dyds \\
    &\lesssim \int_0^t \int_{\dR} G^2(t-s,y-x) \left (n^2 \sigma^2(n^{-1})\| u (s,y) - v(s,y) \|_2^2 + \sigma^2(n^{-1}) \right) dyds,
  \end{aligned}
\end{equation*} for all sufficiently large $n\geq 1 $. As before, we then get 
\begin{equation*}
  \begin{aligned}
    \sup_{x\in \dR} \| u (t,x) - v(t,x) \|_2^2 \lesssim n^2 \sigma^2(n^{-1}) \int_0^t (t-s)^{-1/2} \cdot \sup_{x\in \dR} \| u (s,x) - v(s,x) \|_2^2 ds + \sigma^2(n^{-1}). 
  \end{aligned}
\end{equation*} By the Gr\"onwall inequality, there exists a constant $C>0$ such that 
\begin{equation*}
\begin{aligned}
  \sup_{0\leq t \leq T}\sup_{x\in \dR} \| u (t,x) - v(t,x) \|_2^2  &\leq \sigma^2(n^{-1}) E_{\frac{1}{2}}\left(C n^2 \sigma^2(n^{-1}) T^{1/2}  \right)\\
  &\leq \sigma^2(n^{-1})\exp \left(C^2 n^4 \sigma^4(n^{-1}) T  \right).
\end{aligned}
\end{equation*} By \eqref{eq:bound for the growh of sigma} and the argument below \eqref{eq:bound for the growh of sigma}, we can derive that $\sup_{0\leq t \leq T}\sup_{x\in \dR} \| u (t,x) - v(t,x) \|_2 \to 0 $ as $n\to \infty$. This concludes that $u=v$ in $C([0,T] \times \dR ; L^2(\Omega))$.
\end{proof}

From now on, we assume that $u(t, x)$ is the unique mild solution of \eqref{eq:SHE} from Proposition \ref{prop:existence of strong mild solution}. In the next lemma, we present the moment estimates for the solution $u$ of \eqref{eq:SHE}, which is a key input for the proofs of the H\"older continuity in time and space and also the strict positivity of $u$. 

\begin{lemma}\label{lem:finiteness of moment}  There exist constants $C=C(\|u_0\|_{L^\infty})>0$ and $K=K(L_\fd)>0$ such that for any $T>0$, $p\geq 2 $, and $\delta \in (0, \fd \wedge 1 )$,
\begin{equation}\label{eq:p_th moment of u}
 \sup_{0\leq  t\leq T}\sup_{x\in \dR}\dE [|u(t,x)|^p] \leq C^p \exp\left( K \delta^{-4}\sigma^4(\delta )p^3 T  \right).
\end{equation} 
  \end{lemma}
\begin{proof} 
Let $\Phi:= \{ \Phi(t,x) \}_{(t,x) \in \dR_+\times \dR}$ be a random field, and for every $\alpha >0$ and $p\geq 2$,  define 
\begin{equation}\label{eq:definition of N norm}
  \cN_{\alpha, p}(\Phi): = \sup_{t\geq 0}\sup_{x\in \dR}\left(e^{-\alpha t}\| \Phi(t,x)\|_p \right).
\end{equation} Note that $\cN_{\alpha, p}(u)<\infty$, since by Fatou's lemma and Lemma \ref{lem:uniform moment bound for u^m}, we have that for all $x\in \dR$ 
\begin{equation*}
 \|u(t,x)\|_p \leq \liminf_{m\to \infty} \|u^{(m)}(t,x) \|_p <\infty. 
 \end{equation*} We write $u(t,x) $ as
\begin{equation*}
  u(t,x)  = {\bf I}(t,x)+ \bN(t,x), \quad \text{for all }(t,x)\in \dR_+\times \dR,
\end{equation*} where $ {\bf I}(t,x) : = (G(t,\cdot)*u_0)(x) $ and  $\bN(t,x)$ is defined as 
\begin{equation*}
  \bN(t,x) : = \int_0^t\int_{\dR} G(t-s,y-x)\sigma(u(s,y)) W(dsdy).
\end{equation*} Let us fix an arbitrary $\delta \in (0, \fd \wedge 1 )$. From \eqref{eq:delta-Lipschitz continuity of sigma} with $v=0$, we can see that for all $u\in \dR$
\begin{equation}\label{eq:upper bound of sigma(u)}
  |\sigma(u)| \leq \left| (\delta^{-1} \sigma(\delta) \vee L_\fd) u \right| + \sigma(\delta).
\end{equation} From now on, we assume that $\delta^{-1}\sigma(\delta)\geq L_\fd$, otherwise the proof is much simpler.  For $p\geq 2 $, we apply the BDG inequality and  to get 
\begin{equation}\label{eq:moment bound for the stochastic term}
\begin{aligned}
    \| \bN(t,x) \|_p^2 &\leq C p \int_0^t\int_{\dR} G^2(t-s, y-x) \| \sigma(u(s,y)) \|_2^p dyds\\
    &\leq Cp \int_0^t\int_{\dR} G^2(t-s, y-x) \left( \delta^{-2} \sigma^2(\delta)\| u(s,y) \|_2^p  + \sigma^2(\delta)\right) dyds\\
    &\leq Cp  \delta^{-2} \sigma^2(\delta)\int_0^t\int_{\dR} G^2(t-s, y-x)\| u(s,y) \|_2^p   dyds + C p  \sigma^2(\delta) t^{1/2},
\end{aligned}
\end{equation} for some constant $C=C(L_\fd)>0$. We note that $C$ depends on $L_\fd$ due to \eqref{eq:upper bound of sigma(u)}. On the other hand, there exists $C_1>0$ only depending on $\|u_0\|_{L^\infty}$ such that 
  \begin{equation}\label{eq:estimates for Initial data}
    \| \bI(t,x) \|_p^2 \leq C_1  \quad \text{for all }p\geq2, (t,x) \in \dR_+\times \dR.
  \end{equation}By \eqref{eq:moment bound for the stochastic term} and \eqref{eq:estimates for Initial data}, we see that there exist constants $C'_1= C'_1(\|u_0\|_{L^\infty})>0$ and  $C_2 = C_2(L_\fd)>0$ satisfying 
\begin{equation}\label{eq:estimates for p_th moment of u greater than delta via BDG}
\begin{aligned}
      \sup_{x\in\dR}\|u(t,x)\|_p^2 \leq  C'_1 + C_2p\left(\sigma^2(\delta) t^{1/2} + \delta^{-2}\sigma^2(\delta)\int_0^t \int_{\dR}G^2(t-s,y-x)\sup_{x\in \dR} \|u(s,x)\|_p^2 dyds \right).
\end{aligned}
 \end{equation} For the third term on the RHS, by the definition of $\cN_{\alpha,p}$ (see \eqref{eq:definition of N norm}), there exists a universal constant $C_3>0$ such that 
 \begin{equation*}
 \begin{aligned}
    \int_0^t \int_{\dR}G^2(t-s,y-x)\sup_{x\in \dR} \|u(s,x)\|_p^2 dyds  &\leq \left[\cN_{\alpha,p}( u)\right]^2\int_0^t\int_{\dR} e^{2\alpha s} G^2(t-s,y-x)dyds\\
    &= \left[\cN_{\alpha,p}( u)\right]^2  \int_0^t e^{2\alpha s }(4\pi (t-s))^{-1/2} ds\\
    &= \left[\cN_{\alpha,p}( u)\right]^2 e^{2\alpha t} \int_0^t e^{-2\alpha r}(4\pi r)^{-1/2} dr\\
    & \leq C_3\left[\cN_{\alpha,p}( u)\right]^2 \alpha^{-1/2},
 \end{aligned}
 \end{equation*}  where we have used the following calculations:
 \begin{equation*}
  \int_{\dR} G^2(t-s,y-x)^2 dy = (4\pi (t-s))^{-1/2},
 \end{equation*} and 
 \begin{equation*}
  \int_0^\infty e^{-2\alpha r}(4\pi r)^{-1/2} dr \lesssim \alpha^{-1/2}.
 \end{equation*} Therefore, by multiplying $e^{-2\alpha t}$ and taking the supremum over $t\geq 0$, there exists $C_4=C_4(L_\fd)>0$  such that
 \begin{equation*}
 \begin{aligned}
   \cN_{\alpha,p} (u  ) \leq C'_1 + C_4 p\left(  \alpha^{-1/2}\sigma^2(\delta)  + \alpha^{-1/2}\delta^{-2}\sigma^2(\delta)  \cN_{\alpha,p} (u )  \right). 
 \end{aligned}
 \end{equation*} Choose $\alpha  = K \delta^{-4}\sigma^4(\delta)p^2$ with $K=K(L_\fd)>0$ satisfying 
 \begin{equation*}
  C_4 p\left[ (\alpha^{-1/2}\sigma^2(\delta)) \vee (\alpha^{-1/2}\delta^{-2}\sigma^2(\delta))\right] \leq \frac{1}{4}.
 \end{equation*} We remark that the constant $K$ is independent of $\delta\in(0,\fd\wedge 1 )$ and $p\geq 2 $. This ensures that
 \begin{equation*}
   \cN_{\alpha,p} (u ) \leq C'_1+ \frac{1}{4}\left( 1+  \cN_{\alpha,p} (u )\right),
 \end{equation*} which implies 
\begin{equation*}
    \cN_{\alpha,p} (u) \leq C: =  \frac{4}{3}\left( C'_1 + \frac{1}{4} \right).
  \end{equation*} Now it turns out that for all $p\geq 2 $ and $T>0$,
  \begin{equation*}
    \sup_{0\leq t \leq T}\sup_{x\in \dR}\dE |u(t,x)|^p\leq C^p \exp\left( K \delta^{-4}\sigma^4(\delta )p^3 T  \right).
  \end{equation*} This concludes the proof of \eqref{eq:p_th moment of u}.
\end{proof}

We can now provide the proof of Theorem \ref{thm:strict positivity} (i), which establishes the existence, uniqueness, and H\"older continuity of a mild solution. We will also prove Theorem \ref{thm:weak comparison principle}, which states the weak comparison principle. 

\begin{proof}[\textbf{Proof of Theorem \ref{thm:strict positivity} (i)}]
  The existence and uniqueness of a mild solution follows from Proposition \ref{prop:existence of strong mild solution}. For the H\"older continuity of $u$, the proof follows the same lines of the proof of \cite[Theorem 3.1]{chen2014holder} with Lemma \ref{lem:finiteness of moment} and the fact that $\dE[\sigma^2(u(t,x))] \leq C(1+ \dE[|u(t,x)|^2]) $. 
\end{proof}

\begin{proof}[\textbf{ Proof of Theorem \ref{thm:weak comparison principle}}]
By Proposition \ref{prop:existence of strong mild solution}, we see that $u^{(m,1)}(t,x) \to u^{(1)}(t,x)$ and $u^{(m,2)}(t,x) \to u^{(2)}(t,x)$ in $L^2(\Omega)$ for each $(t,x)\in \dR_+\times \dR$ with initial data $u^{(1)}_0$ and $u^{(2)}_0$ respectively. In addition, $u^{(1)}$ and $u^{(2)}$ are both continuous in $\dR_+\times \dR$ by Theorem \ref{thm:strict positivity} (i). Thus, it is clear that Lemma \ref{lem:weak comparison for u^m} completes the proof of Theorem \ref{thm:weak comparison principle}.
\end{proof}

\section{Strict positivity}\label{sec:strict positivity}
In this section, our aim is to prove Theorem \ref{thm:strict positivity} (ii). We start by relaxing the problem into simpler setting. We claim that to prove strict positivity, it is sufficient to consider the solution $u$ of the SPDE
\begin{equation}\label{eq:SPDE for strict positivity}
  \begin{aligned}
    \begin{cases}
      \partial_tu(t,x) = \frac{1}{2}\partial^2_x u(t,x) + \sigma(u(t,x))\dot{W}(t,x), \quad(t,x)\in \dR_+ \times \dR,\\
      u(0,x) = \1_{[-r,r]}(x), \quad x\in\dR, 
    \end{cases}
  \end{aligned} 
\end{equation} where $r>0$. Indeed, since we assumed that $u_0$ is a nonnegative function in $C(\dR)$ which is not identically zero, by the weak comparison principle (Theorem \ref{thm:weak comparison principle}), it suffices to consider $u_0(x) = c\1_{[a,b]}$ for any $a,b\in \dR$ and $c>0$. Appealing to the translation $u_0(\cdot)\mapsto u_0(\cdot-(a+b)/2)$, we can assume that $u_0(x) = c\1_{[-r,r]}$ for some $r>0$. Moreover, we can consider $\bar{u}(t,x) : = c\cdot u(t,x)$ which is the unique mild solution of \eqref{eq:SHE} with the initial function $\1_{[-r,r]}$ and with the nonlinearity $c\cdot \sigma(c^{-1}x)$ instead of $\sigma(x)$. Observe that $c \cdot \sigma(c^{-1}\cdot)$ satisfies Assumption \ref{assumption:basic condition for sigma} and \eqref{eq:positivity_condition} if $\sigma$ does, with $c\cdot \fd $ instead of $\fd$. In particular, we emphasize that Assumption \ref{assumption:basic condition for sigma} (b) holds with $L_{c\fd} = L_\fd$. Thus, in the rest of the section we only consider the solution $u$ of \eqref{eq:SPDE for strict positivity} with the initial function of the form $u_0 = \1_{[-r,r]}$ and $\sigma$ satisfying Assumption~\ref{assumption:basic condition for sigma} and \eqref{eq:positivity_condition}.

\subsection{Probability estimates for strict positivity}

In what follows, we present the main contribution of this section, which will be iteratively used to prove Theorem \ref{thm:strict positivity} (ii). The following proposition provides the lower tail estimate for the unique mild solution $u$ to \eqref{eq:SPDE for strict positivity} with the initial function $u_0 =\1_{[-r,r]} $ for $r>0$, which is a generalization of similar estimates in \cite{chen2017comparison,mueller1991support,shiga1994two} into a larger class of non-Lipschitz nonlinearity $\sigma$. The main input of the proof is the moment estimate in Lemma \ref{lem:finiteness of moment} (see the proof of Lemma \ref{lem:estimation of stochastic term}).

\begin{proposition}\label{prop:key prop for strict positivity} Fix $\fr>0$. Let $u$ be the unique mild solution to \eqref{eq:SPDE for strict positivity} with an initial condition $u_0=\1_{[-r,r]}$ and satisfy Assumption~\ref{assumption:basic condition for sigma} and \eqref{eq:positivity_condition}. For any $T>0$, $M>\fr$, $\eta\in(0,1/4)$, and $\delta \in(0,\fd\wedge 1 )$, there exists $C = C(T,M, L_\fd,\eta)>0 $ and $m_0 = m_0(T,M,L_\fd,\eta,\delta,\fr)>0$ such that for all $m\geq m_0$ and $r\in[\fr,2M]$
  \begin{equation}\label{eq:key probability estimate for stric positivity}
    \begin{aligned}
        \dP&\left(u(s,x) \geq \eta \1_{[-r-M/m, r+ M/m]}(x) \text{ for all }(s,x) \in \left[ \frac{T}{2m}, \frac{T}{m}\right]\times \dR \right) \\
        &\quad \geq 1-  \exp \left\{-C m^{1/2}(\log m )^{3/2}\left(\frac{\delta }{\sigma(\delta)} \right)^2  \right\}.
    \end{aligned}      
   \end{equation} 
\end{proposition}
\begin{remark}\label{rmk:uniform_in_r} We point out that the constants $C$ and $m_0$ in Proposition \ref{prop:key prop for strict positivity} are independent of $r\in[\fr,2M]$ while $m_0$ depends on the lower bound $\fr$. The uniform-in-$r$ bound \eqref{eq:key probability estimate for stric positivity} is crucial since we will exploit Proposition \ref{prop:key prop for strict positivity} repeatedly with $r=\fr + \frac{kM}{m}$ for $k=0,1,...,m-1$ and $m\geq m_0$ instead of a fixed $r$ (see Section \ref{subsec:proof of positivity}).

\end{remark}
To prove Proposition \ref{prop:key prop for strict positivity}, we consider the unique mild solution of \eqref{eq:SPDE for strict positivity}, as in the previous section. Specifically, we work with the following integral equation for $u$:
\begin{equation*}
  u(t,x)  = (G(t,\cdot)*\1_{[-r,r]})(x) + \bN(t,x), \quad \text{for all }(t,x)\in \dR_+\times \dR,
\end{equation*} where the stochastic term $\bN(t,x)$ is defined as 
\begin{equation}\label{eq:noise} 
  \bN(t,x) : = \int_0^t\int_{\dR} G(t-s,y-x)\sigma(u(s,y)) W(dsdy).
\end{equation} The following two lemmas are essential to prove Proposition \ref{prop:key prop for strict positivity}.

\begin{lemma}\label{lem:lower bound for the heat propagation} 
Fix $\fr>0$. Let $T,M>0$ and $\eta\in(0,1/4)$. There exists a constant $m_0=m_0(T,M,\eta,\fr)\geq 1 $ such that for all $m\geq m_0$ and $r\geq \fr$ we have 
\begin{equation}\label{eq:lower bound for the heat propagation}
  \inf_{s\in[T/(2m),T/m ]}(G(s,\cdot) * \1_{[-r,r]})(x) \geq  2\eta\1_{[-r -M/m, r+ M/m]}(x) \quad \text{for all } x\in \dR.
\end{equation} 
\end{lemma}
\begin{proof}
  The proof follows similarly as in the proof of \cite[Lemma 4.1]{chen2017comparison}. Let $Z$ be a standard normal distribution with the density $G(1,x)$. For the case where $-r-M/m\leq x \leq 0$, for all $ s\in[T/(2m),T/m ]$, we have that whenever $m>\sqrt{2}M/\fr$, 
  \begin{equation*}
    \begin{aligned}
      (G(s,\cdot) * \1_{[-r,r]})(x) &= \int_{-r}^r G(s,x-y) dy \\
      &= \int_{(x-r)/\sqrt{s}}^{(x+r)/\sqrt{s}} G(1, y) dy  \\
      &\geq \dP \left( -\frac{r\sqrt{m}}{\sqrt{T}} \leq Z \leq -\frac{\sqrt{2}M}{\sqrt{mT}} \right).
    \end{aligned}
  \end{equation*} Similarly, in the case $0\leq x\leq r+  M/m$, we have 
  \begin{equation*}
    \begin{aligned}
            (G(s,\cdot) * \1_{[-r,r]})(x)  &\geq \dP \left(\frac{\sqrt{2}M}{\sqrt{mT}} \leq  Z \leq \frac{r\sqrt{m}}{\sqrt{T}} \right).
    \end{aligned}
  \end{equation*} Note that as $m\to \infty$, 
\begin{equation*}
  \dP \left( -\frac{\fr\sqrt{m}}{\sqrt{T}} \leq Z \leq -\frac{\sqrt{2}M}{\sqrt{mT}} \right) \wedge \dP \left(\frac{\sqrt{2}M}{\sqrt{mT}} \leq  Z \leq \frac{\fr\sqrt{m}}{\sqrt{T}} \right)\to \frac{1}{2}.
\end{equation*}
   Therefore, due to $r\geq \fr$, we can choose $m_0=m_0(T,M,\eta , \fr)\geq 1 $ large so that \eqref{eq:lower bound for the heat propagation} holds for all $m\geq m_0$.
\end{proof}

\begin{lemma}\label{lem:estimation of stochastic term} Fix $\fr>0$. Let $T,\eta>0$, $M>\fr$ and $\delta \in(0,\fd\wedge 1 )$.  Then, there exist some constants $C = C(T,M, L_\fd,\eta)>0 $ and $m_0 = m_0(T,M,L_\fd, \eta,\delta)>0$ such that for all  $m\geq m_0$  and $r\in[\fr, 2M]$ 
\begin{equation*}
  \dP \left( \sup_{\substack{(s,y)\in S } } |\bN(s,x)| \geq \eta \right) \leq\exp \left\{-C m^{1/2}(\log m )^{3/2}\left(\frac{\delta }{\sigma(\delta)} \right)^2  \right\},
\end{equation*} where $S : = \left\{ (s,y) : 0\leq s\leq T/m, |y| \leq r+  M/m \right\}$.
\end{lemma}
\begin{proof} We begin the proof by applying the BDG inequality, \cite[Proposition~4.4]{chen2014holder}, and Lemma \ref{lem:finiteness of moment} to get that  there exists $C_1>0$ such that for all $p\geq 2 $, $\delta \in(0,\fd\wedge 1 )$ and $(s,y),(s',y')\in S$, 
\begin{equation*}
\begin{aligned}
    \dE[|\bN(s,y) - \bN(s',y') |^p] &\leq C_1^p ( | y-y'|^{1/2} + |s-s'|^{1/4})^p  \cdot \sup_{(t,x)\in S}\| \sigma(u (t,x)) \|_p^p \\
      &\leq C_1^p \bM_{T,p,\delta} ( | y-y'|^{1/2} + |s-s'|^{1/4})^p, \end{aligned}
\end{equation*}  where 
\begin{equation*}
  \bM_{T,p,\delta} : =  C^p \exp \left\{ Kp^3 \left( \frac{T}{m}\right) \left( \frac{\sigma(\delta) }{\delta}\right)^4\right\}.
\end{equation*} with some constant $C>0$ and the constant $K=K(L_\fd)>0$ in Lemma \ref{lem:finiteness of moment}. Here, we used the fact that $\| \sigma(u (t,x)) \|_p^p \leq  2^p{L_{\fd}}^p (1 +\| u (t,x) \|_p^p) $.  On the other hand, for all $\theta \in  (0 , 1-6p^{-1})$ and $p>6$,  we have
\begin{equation*}
  \begin{aligned}
    \dE\left[ \sup_{(s,y)\in S} \left|\frac{ \bN (s,y)}{(T/m)^{\theta/4}} \right|^p \right] &\leq \dE\left[ \sup_{(s,y)\in S} \left|\frac{ \bN (s,y) - \bN(0,y)}{s^{\theta/4}} \right|^p \right]\\
    &\leq \dE\left[ \sup_{\substack{(s,y),(s',y')\in S\\  (s,y)\neq (s',y')}} \left|\frac{ \bN (s,y) - \bN(s',y')}{(|y-y'|^{1/2} + |s-s'|^{1/4})^{\theta}} \right|^p \right],
  \end{aligned}
\end{equation*} where we have used that $\bN(0,x) = 0$ for all $x\in\dR$. Note that  $S\subset [0,3M]^2 $ for all $m\geq 1 $ and $r\in[\fr,2M]$. Then, by the quantitative form of the Kolmogorov continuity theorem (see e.g. \cite[Theorem C.6]{khoshnevisan2014analysis}), we find that there exists $C_2=C_2(M,L_\fd)> 0 $  such that for all $p>6$ and $m\geq 1$
\begin{equation}\label{eq:moment estimate of supremum of N}
  \dE\left[ \sup_{(s,y)\in S} \left| \bN (s,y) \right|^p \right] \leq C_2^p(T/m )^{\frac{\theta p }{4}}\exp \left\{ Kp^3 \left( \frac{T}{m}\right) \left( \frac{\sigma(\delta) }{\delta}\right)^4\right\}.
\end{equation} This implies that for all $\eta>0$,
\begin{equation}\label{eq:chebyshev ineq for stochastic term}
\begin{aligned}
  \eta^{-p} \dE \left[ \sup_{(s,y)\in S} \left| \bN (s,y) \right|^p \right] &\leq \eta^{-p}C_2^p (T/m )^{\frac{\theta p }{4}} \exp \left\{ Kp^3 \left( \frac{T}{m}\right) \left( \frac{\sigma(\delta) }{\delta}\right)^4\right\}  \\
  &\leq \exp\left\{ p \left( \log C_2+ \log (\eta^{-1})  - \frac{\theta  }{4} \log \frac{m}{T} \right)+ K p^3  \left( \frac{\sigma(\delta) }{\delta}\right)^4 \frac{T}{m} \right\} .
\end{aligned}
\end{equation} We can choose a sufficiently large $m_1:=m_1(T,M,L_\fd , \eta)\geq 1 $ such that for all $m\geq m_1$ 
\begin{equation*}
  \log C_2+ \log (\eta^{-1})  - \frac{\theta  }{4} \log \frac{m}{T} < 0.
\end{equation*} Notice that we fixed an arbitrary $\theta \in (0,p)$ for some $p>6$ which will be chosen below , thus we drop the dependence on $\theta$.  We find that the function $ p \mapsto f(p) :  = ap^3 - bp $ for $a,b>0$ is minimized at $p_0= \sqrt{b/(3a)}$ and $f(p_0) = -2 (b/3)^{3/2}a^{-1/2}$. Keep this in mind and let 
\begin{equation*}
  p_1= \left( \frac{\frac{\theta  }{4} \log  (m/T)-\log C_2-\log (\eta^{-1})  }{ 3K   \left( \frac{\sigma(\delta) }{\delta}\right)^4 \frac{T}{m} }\right)^{1/2}.
\end{equation*} We can take $m_2= m_2(T,M,L_\fd,\eta,\delta)>0 $  that guarantees $p_1>6$ for all $m\geq m_2$. With this choice of $p=p_1$ and by the Chebyshev inequality, we obtain from \eqref{eq:chebyshev ineq for stochastic term} that there exists  $C = C(T,M,\eta, L_\fd)>0 $ such that 
\begin{equation*}
  \begin{aligned}
    \dP \left( \sup_{\substack{(s,y)\in S } } |\bN(s,x)| \geq \eta \right)\leq \eta^{-p} \dE \left[ \sup_{(s,y)\in S} \left| \bN (s,y) \right|^p \right] &\leq \exp \left\{-C m^{1/2}(\log m )^{3/2}\left(\frac{\delta }{\sigma(\delta)} \right)^2  \right\}.
  \end{aligned}
\end{equation*} We note that $C$ and $m_2$ depend on $L_\fd$ since $K$ does. This completes the proof.

\end{proof}
  
The proof of Proposition \ref{prop:key prop for strict positivity} now follows from Lemmas \ref{lem:lower bound for the heat propagation} and \ref{lem:estimation of stochastic term}. 

\begin{proof}[\textbf {Proof of Proposition \ref{prop:key prop for strict positivity}}]
   Thanks to Lemma \ref{lem:lower bound for the heat propagation}, for any  $\eta \in (0,1/4)$ and all $x\in \dR$
  \begin{equation*}
    \inf_{s\in[T/(2m),T/m ]}(G(s,\cdot) * \1_{(-\fr,\fr)})(x) \geq 2\eta\1_{(-r -M/m , r+ M/m)}(x).
  \end{equation*} Define 
  $$S: = \{ (t,x) : T/(2m) \leq t \leq T/m , |x|\leq r + M/m\}.$$ 
  Then, we have 
  \begin{equation*}
    \begin{aligned}
      \dP&\Big(u(s,x) <  \eta \1_{(-r-M/m, r+ M/m)}(x) \text{ for some }(s,x) \in [T/(2m),T/m ]\times \dR \Big)\\&\leq \dP\left( \sup_{(s,x)\in S} |\bN(s,x)| \geq \eta\right). 
    \end{aligned}
  \end{equation*} By Lemma \ref{lem:estimation of stochastic term}, we obtain \eqref{eq:key probability estimate for stric positivity}.

\end{proof}

\subsection{Proof of strict positivity: Theorem \ref{thm:strict positivity} (ii)}\label{subsec:proof of positivity}

In this section, we prove Theorem \ref{thm:strict positivity} (ii) (strict positivity of the solution to \eqref{eq:SHE}). 

\begin{proof}[\textbf{Proof of Theorem \ref{thm:strict positivity} (ii)}]

 Let $T>0$ and $M>\fr>0$ be fixed but arbitrary constants. We will prove that 
\begin{equation}\label{eq:probability_of_strict_positivity}
  \dP\left( u(s,x) > 0 \text{ for all }(s,x) \in \left[\frac{T}{2}, T \right] \times\left[ -\frac{M}{2},\frac{M}{2} \right]  \right)=1,
\end{equation} where $u$ is the unique mild solution to \eqref{eq:SPDE for strict positivity} subject to the initial function $u_0 = \1_{[-\fr,\fr]}$ for any fixed $\fr>0$. Since our choices of $T>0$ and $M>\fr>0$ are arbitrary, by the weak comparison principle (Theorem~\ref{thm:weak comparison principle}) with the discussion at the beginning of Section~\ref{sec:strict positivity}, \eqref{eq:probability_of_strict_positivity} is sufficient for the proof of Theorem~\ref{thm:strict positivity} (ii). 

Let $m_0\geq 1 $ be a large integer which will be specified later (see \eqref{eq:lower bound for the probability of w_k} below). For any integer $m\geq m_0$, we define for $k=0,1,...,m-1$ 
\begin{equation*}
  S^m_k : = \left\{ x\in \dR : |x|\leq \fr + \frac{Mk}{m} \right\}.
\end{equation*}  For $k\geq 1$, let  $w_k$ be the solution to the following SPDE:
\begin{equation}\label{eq:SPDE for iteration}
  \begin{aligned}
     \begin{cases}
      \partial_tw_k(t,x) = \frac{1}{2}\partial^2_x w_k(t,x) + \sigma_k(w_k(t,x))\dot{W}_k(t,x), \quad(t,x)\in \dR_+ \times \dR,\\
      w_k(0,x) = \mathds{1}_{S^m_{k-1}}(x), \quad x\in\dR,
     \end{cases}
  \end{aligned}
 \end{equation}  where $\dot{W}_k(s,x) : = \dot{W}(s+ kT/m , x)$ is the time-shifted space-time white noise. Fix $\eta \in(0,1/4\wedge \fd)$. We set
 \begin{equation*}
  \sigma_k(u) : = \eta^{-k}\sigma(\eta^k u) \quad \text{for all }u\in \dR.
 \end{equation*} 
 Let $v_k(t,x)$ be a mild solution to \eqref{eq:SHE} driven by the time-shifted noise $\dot{W}_k$ starting from $v_k(0,x) : = \eta^{k} \1_{S^m_{k-1}}$.

Now let us describe our strategy for proving \eqref{eq:probability_of_strict_positivity}. First note that $u(s+kT/m, x) \geq v_k(s,x)$ for all $s\geq 0 $ and $x\in\dR$, conditioned on $u(kT/m,x)\geq v_k(0,x) =  \eta^k \1_{S^m_{k-1}}(x)$ for all $x\in\dR$ by the weak comparison principle. Thus, we aim to show that $v_k$ is strictly positive with high probability in each time interval of length $T/(2m)$ inductively for all $k=0,1,...,m-1$ (see \eqref{eq:lower bound for the probability of w_k} below) by applying Proposition~\ref{prop:key prop for strict positivity}. However, since the size of the initial condition of $v_k$ (i.e. $\eta^k$) is not fixed as $k$ varies, and the constants in Proposition~\ref{prop:key prop for strict positivity} depend on it, we instead apply the proposition to the solution $w_k$ of \eqref{eq:SPDE for iteration} which has a fixed-size initial condition $\1_{S^m_{k-1}}$.

To this end, we first prove that $v_k = \eta^k w_k$ by uniqueness (Theorem \ref{thm:strict positivity} (i)). Thus, we verify that Assumption \ref{assumption:basic condition for sigma} is satisfied by $\sigma_k$: The monotonicity of $u\mapsto \sigma_k(u)$ and $u\mapsto \sigma_k(u)/u$ follows from that of $\sigma $ with $\fd$ replaced by $\fd_k := \eta^{-k}\fd$. In addition, $\sigma_k(u) \leq L_{\fd_k} u$ with $L_{\fd_k} = L_{\fd}$. Moreover, it is easy to see that the condition \eqref{eq:positivity_condition} also holds for $\sigma_k$ uniformly over all $k$. Therefore, by Theorem \ref{thm:strict positivity} (i), there exists a unique mild solution $w_k$ of \eqref{eq:SPDE for iteration}, and by uniqueness we have that $v_k(t,x) = \eta^k w_k(t,x)$. 

 Next we apply Proposition~\ref{prop:key prop for strict positivity} to $w_k$ for all $k=0,1,...,m-1$. Regarding the assumptions in Proposition \ref{prop:key prop for strict positivity}, we find that replacing $r$ by $\fr+(Mk)/m$ is valid since we have chosen $M>\fr$ to ensure that $r= \fr+(Mk)/m\in[\fr,2M]$ (see Remark~\ref{rmk:uniform_in_r}), and $\eta < \fd \wedge 1\leq \fd_k \wedge 1 $. Then by Proposition \ref{prop:key prop for strict positivity}, there exist constants $C=C(T,M,L_\fd,\eta) \geq 1 $ and $m_0 = m_0(T,M,L_\fd,\eta,\fr)\geq 1 $ such that for all integers $m\geq m_0$ and $k=0,1,...,m-1$ we have 
\begin{equation}\label{eq:lower bound for the probability of w_k}
\begin{aligned}
  \dP&\left(v_k(s,x) \geq \eta^{k+1} \1_{S^m_k}(x) \text{ for all }(s,x) \in \left[ \frac{T}{2m},\frac{T}{m}\right]\times \dR \right) \\
  & = \dP\left(w_k(s,x) \geq \eta \1_{S^m_k}(x) \text{ for all }(s,x) \in \left[ \frac{T}{2m},\frac{T}{m}\right]\times \dR \right)\geq 1 - g(k),
\end{aligned}
 \end{equation} where $g(k)$ is defined as
 \begin{equation*}
 \begin{aligned}
    g(k) &: = \exp \left( -C   m^{1/2}(\log m )^{3/2}\left[\frac{\eta }{\sigma_k(\eta )} \right]^2\right) \\
  &= \exp \left( -C  m^{1/2}(\log m )^{3/2}\left[\frac{\eta^{k+1} }{\sigma(\eta^{k+1} )} \right]^2\right).
 \end{aligned} 
 \end{equation*} By Assumption \ref{assumption:basic condition for sigma} (a), uniformly over all $k=0,1,...,m-1$, we have 
 \begin{equation}
  \fg(m) : = \exp \left( -C  m^{1/2}(\log m )^{3/2}\left[\frac{\eta^m }{\sigma(\eta^m)} \right]^2\right) \geq g(k).
 \end{equation} Furthermore, by \eqref{eq:positivity_condition}, there exists a constant $C=C(\alpha,\eta)>0$ such that 
\begin{equation*}
  \frac{\eta^m}{\sigma(\eta^m)} = \frac{e^{-m\log\eta^{-1} }}{\sigma(e^{-m\log\eta^{-1}})}\geq  Cm^{-\alpha},
 \end{equation*} for all sufficiently large $m\geq 1 $. This leads to an upper bound of $\fg(m)$ as follows: There exist constants $C_1=C_1(T,M,L_{\fd},\eta)>0$ and  $m_1= m_1(T,M,L_{\fd},\eta,\fr,\alpha)>0$ such that for all $m\geq m_1$
 \begin{equation}
  \fg(m) \leq \exp \left(  -C_1  m^{1/2}(\log m )^{3/2} m^{-2\alpha}\right) = \exp (-C_1 m^{\bar{\alpha}} (\log m)^{3/2}),
 \end{equation} where $\bar{\alpha} = \frac{1}{2} - 2\alpha >0$. 
 
 We define the events  
\begin{equation*}
  \begin{aligned}
    &{\bf A}_k : = \left\{u(s,x) \geq \eta^{k+1}  \1_{S^m_k}(x) \text{ for all }(s,x) \in \left[ \frac{(2k+1)T}{2m}, \frac{(k+1)T}{m} \right]\times \dR \right\}, \text{ for $k=0,1,...,m-1$}, \\
    &{\bf B}_k : = \left\{u(s,x) \geq \eta^{k+1}  \1_{S^m_k}(x) \text{ for all }(s,x) \in \left[ \frac{kT}{m} ,\frac{(2k+1)T}{2m}\right]\times \dR \right\}, \text{ for $k= 1,\dots , m-1$}, \\
    &{\bf B}_0 : = \left\{u\left( \frac{T}{2m}, x\right) \geq \eta \1_{S^m_0}(x) \text{ for all }x\in \dR \ \right\}.
  \end{aligned} 
\end{equation*} Note that the decomposition of the events into ${\bf A_k}$ and ${\bf B_k}$ is required since we have to wait for an intermidiate time of length $T/2M$ to apply Proposition~\ref{prop:key prop for strict positivity} iteratively. Then, by \eqref{eq:lower bound for the probability of w_k} with the weak comparison principle (Theorem \ref{thm:weak comparison principle}) and the Markov property of $u$ (see \cite[Section 9]{da2014stochastic}) we have that for all $1\leq k \leq m-1$
\begin{equation*}
  \dP ( {\bf A}_k \mid \cF_{kT/m}) \geq 1 - g(k) \geq 1 - \fg(m) ,\quad \text{a.s. on }{\bf A}_{k-1},
\end{equation*} since $u(s+kT/m,x) \geq v_k(s,x)$ for all $(s,y)\in\dR_+\times \dR$, a.s. on ${\bf A}_{k-1}$. We also have 
\begin{equation*}
  \dP ( {\bf A}_k \mid {\bf A}_{k-1} \cap \cdots \cap {\bf A}_0) \geq 1 - \fg(m).
\end{equation*} Similarly, for $1\leq k \leq m-1$ we have 
\begin{equation*}
  \dP ( {\bf B}_k \mid {\bf B}_{k-1} \cap \cdots \cap {\bf B}_0) \geq 1 - \fg(m).
\end{equation*} and $\dP({\bf B}_0) \geq\dP({\bf A}_0)\geq 1 - \fg(m)  $. Therefore, we get 
 \begin{equation*}
 \begin{aligned}
    \dP \left( \bigcap_{0\leq k \leq m-1 } [{\bf A}_k \cap {\bf B}_k]  \right) &\geq 1 - \left( 1- \dP \left( \bigcap_{0\leq k \leq m-1 } {\bf A}_k \right)\right)-\left( 1- \dP \left( \bigcap_{0\leq k \leq m-1 } {\bf A}_k \right)\right)\\
  &\geq (1-\fg(m))^{m-1} \dP ( {\bf A}_0) + (1-\fg(m))^{m-1} \dP({\bf B}_0) -1 \\
  &\geq 2( 1- \fg(m))^m -1.
 \end{aligned}
 \end{equation*} Let us define 
\begin{equation*}
  {\bf A} : = {\bf A}_{T,M}:=  \left\{ u(s,x) > 0 \text{ for all }(s,x) \in \left[\frac{T}{2}, T \right] \times\left[ -\frac{M}{2},\frac{M}{2} \right]  \right\}. 
 \end{equation*} Now it holds that 
 \begin{equation*}
  \begin{aligned}
    \dP({\bf A})    &\geq \lim_{m\to \infty}\dP \left( \bigcap_{0\leq k \leq m-1 } [{\bf A}_k \cap {\bf B}_k]  \right)\geq \lim_{m\to\infty} 2( 1- \fg(m))^m -1.
  \end{aligned}
 \end{equation*} The proof of \eqref{eq:probability_of_strict_positivity} is completed if we can show 
\begin{equation*}
  \lim_{m\to \infty}( 1- \fg(m))^m =1.
\end{equation*} This follows immediately from the definition of $\fg(m)$ and L'Hopital's rule.
\end{proof}

\section{Compact support property}
\label{sec:proof of compact support property}

The goal of this section is to prove Theorem \ref{thm:cpt_support}, demonstrating the compact support property of the solution to \eqref{eq:SHE} under the assumption \eqref{eq:csp_condition}. The main ingredient for the proof is Lemma \ref{lem:key estimate for CSP}, which states  a stochastic version of the maximum principle. We first provide the proof of Theorem \ref{thm:cpt_support} using Lemma \ref{lem:key estimate for CSP} and then prove Lemma \ref{lem:key estimate for CSP}. Throughout this section, $u$ denotes a nonnegative weak solution of \eqref{eq:SHE} under the assumptions in Theorem \ref{thm:cpt_support}.

\subsection{Proof of compact support property: Theorem \ref{thm:cpt_support}}

For $a >0$, define 
\begin{equation}\label{eq:definition of Psi_a}
  \Psi_a(x) : = \frac{1}{\cosh(ax)}.
\end{equation} The follwing lemma is useful to control the $L^1([0,\tau])$-norm of the solution. 
\begin{lemma}[Lemma 3.1 of \cite{han2023compact}]\label{lem:limit of the time L1 norm} Let $T>0$ and let $\tau\leq T$ be a bounded stopping time. Let $u\in C([0,\tau];C_{tem}(\dR))$ be a nonnegative weak solution to \eqref{eq:SHE} satisfying the assumptions in Theorem~\ref{thm:cpt_support}. Suppose that there exist constants $a,H>0$ satisfying
\begin{equation*}
\sup_{t\leq\tau}\sup_{x\in \dR} \Psi_a(x)u(t,x)\leq H \quad \text{almost surely},
\end{equation*}
where $\Psi_a$ is defined  in \eqref{eq:definition of Psi_a}. Then,  we have
\begin{equation}
\label{eq:limit of the time L1 norm}
\limsup_{R\to\infty}\dE\left[\int_0^\tau  u(t,R)   dt\right]= 0.
\end{equation}
\end{lemma}

We now introduce the key lemma (Lemma \ref{lem:key estimate for CSP} below) which will be used in the proof of Theorem \ref{thm:cpt_support}. The lemma will be proven in the next section. We denote by $I_R$ an interval $[-R,R]$ for $R>0$.

\begin{lemma}
\label{lem:key estimate for CSP} Let $T>0$ and let $\tau\leq T$ be a bounded stopping time. Let $u\in C([0,\tau];C_{tem}(\dR))$ be a nonnegative weak solution to \eqref{eq:SHE} satisfying the assumptions in Theorem~\ref{thm:cpt_support}. Suppose that there exist constants $a, H>0$ such that 
\begin{equation}
\label{eq:bound for the Holder coefficient of u}
\| \Psi_a u\|_{C^\gamma([0,\tau]\times \dR)}\leq H \quad \text{almost surely},
\end{equation} 
where $\Psi_a$ is defined in \eqref{eq:definition of Psi_a}. Furthermore, assume that $\supp(u_0)\subseteq I_{R_0}$ for some $R_0>0$.
Then, the following hold:

\begin{itemize}
\item[(i)] \label{item:key estimate for CSP_1} For every  $R>R_0 \vee 1 $,
\begin{equation*}
\dP \left(  \int_0^\tau    u(s,R)  ds = 0 \right) 
=\dP \Big( u(s,y) =0 \quad\text{for all } s\in[0,\tau] \text{ and }  y>R \Big). 
\end{equation*} 

\item[(ii)] \label{item:key estimate for CSP_2} Let $\lambda= \gamma/(1+\gamma) \in (0,1)$. For every $R>R_0 \vee 1 $, $p,q\in(0,\fd^{1/\lambda} \wedge 1 )$ , $r>0$, and $\beta\in (0,1)$, there exists  $ x \in (r,2r)$  such that 
\begin{equation}
\label{eq:ito inequality for CSP}
\begin{aligned}
\dP& \left( \int_0^{\tau} \sigma(u(s,R+x)) ds \geq p \right) 
\leq \dP\left( \int_0^{\tau}   \sigma(u(s,R))  ds \geq q \right) + C   r^{-\frac{3\beta}{2}} \left(\frac{q^{1+\lambda} }{p\sigma(q^\lambda)}\right)^{\beta },
\end{aligned}
\end{equation}
where $C=C(\beta,\gamma, H,T)>0$. 
\end{itemize}

\end{lemma}

\begin{remark}\label{remark:negative part of key estimate for CSP}
  One can easily show that the results of Lemma \ref{lem:key estimate for CSP} also hold for the negative real line with appropriate modifications. Specifically, the following inequalities hold:
  \begin{equation*}
    \dP \left(  \int_0^\tau    u(s,-R)  ds = 0 \right) 
\leq\dP \Big( u(s,y) =0 \quad\text{for all } s\in[0,\tau] \text{ and }  y<-R \Big),
  \end{equation*} and 
  \begin{equation*}
    \dP \left( \int_0^{\tau} \sigma(u(s,-R-x)) ds \geq p \right) 
\leq \dP\left( \int_0^{\tau}   \sigma(u(s,-R))  ds \geq q \right) + C  r^{-\frac{3\beta}{2}} \left(\frac{q^{1+\lambda} }{p\sigma(q^\lambda)}\right)^{\beta }.
  \end{equation*} respectively.
\end{remark}

We now present the proof of Theorem \ref{thm:cpt_support}.

\begin{proof}[\textbf {Proof of Theorem \ref{thm:cpt_support}}] 
The proof of Theorem \ref{thm:cpt_support} (i) is precisely the same as that of \cite[Theorem 2.9]{han2023compact}. Indeed, Assumption \ref{assumption:basic condition for sigma} implies that we have a universal constant $L>0$ such that $|\sigma(u)| \leq L(1+|u|)$. Since $\sigma$ is continuous, we can follow the same line in \cite[Appendix]{han2023compact}. 

In the remainder of the proof, we prove Theorem \ref{thm:cpt_support} (ii). Fix any $a>0$ and define a stopping time 
\begin{equation*}
\tau_n := n \,\wedge\, \tau \wedge \, \inf\left\{ t\geq0: \|\Psi_a u\|_{C^{\gamma}([0,t]\times\dR)} \geq n \right\}.
\end{equation*} By Theorem \ref{thm:cpt_support} (i), there exists $n_0=n_0(\omega)>0$ almost surely such that $\tau_n = \tau$ for all $n\geq n_0$. Henceforth, throughout the proof, we can assume that there exists a constant $H>0$  satisfying
$$ \tau \leq H \quad\text{and} \quad  \|\Psi_a u\|_{C^\gamma([0,\tau]\times\dR)} \leq H,
$$  and it is sufficient for the proof of Theorem \ref{thm:cpt_support}. For $R > 0$, we set 
\begin{equation*}
\Omega_R := \{ \omega\in\Omega : u(\omega , t,x) = 0\quad\text{for all }t\leq \tau,\, x > R \}.
\end{equation*}
In order to prove the theorem, we aim to show that $\dP(\cup_R \Omega_R) = 1$ by employing Lemma \ref{lem:key estimate for CSP}. Once we prove that $\dP(\cup_R \Omega_R) = 1$, we can use the same argument along with Lemma \ref{lem:key estimate for CSP} and Remark \ref{remark:negative part of key estimate for CSP} to show that
\begin{equation*}
\dP \Big( \text{there exists $R>0$ such that }u(t,x) =0\text{ for all }t\leq \tau \text{ and } x< -R \Big)=1.
\end{equation*}
Moreover, since $\Omega_{R_0}\subset \Omega_{R_1}$ for $R_0 \leq R_1$, it suffices to prove that
\begin{equation*}
\lim_{R\to\infty}\dP(\Omega_R) = 1.
\end{equation*} By Lemma \ref{lem:key estimate for CSP} (i), we have
\begin{equation*}
\dP\left(   \int_0^{\tau}   u(s,R)    ds = 0 \right) \leq \dP(\Omega_R).
\end{equation*} From Assumption \ref{assumption:basic condition for sigma}, we observe that 
\begin{equation*}
  \dP\left( \int_0^{\tau} \sigma(u(s,R))   ds = 0 \right)  \leq \dP\left(  u(s,R) = 0 \,\,\text{for all}\,\, s\leq \tau \right) \leq \dP\left( \int_0^{\tau} u(s,R)   ds = 0 \right). 
\end{equation*} Therefore, the proof is completed once we prove   
\begin{equation}\label{eq: claim in lemma}
\limsup_{R\to\infty} \dP\left( \int_0^{\tau} \sigma(u(s,R))   ds >0 \right) = 0.
\end{equation}

 Let $m\in \dZ_+$ be a sufficiently large positive integer such that $e^{-m} < \fd\wedge 1 $. Set 
\begin{equation*}
  r_k := \frac{1}{2(k+1)^{1+c/3}} \quad \text{ and }\quad p_k : = e^{-\frac{k+m}{\bar{\gamma}}} \quad \text{for all } k\in \dZ_+,
\end{equation*} where $c>0$. Here the choice of $c$ will be specified later.  Let us fix $\beta \in (0,1)$ and $x_0 := R_0 +1$ where $R_0>0$ satisfies $\supp(u_0) \subseteq I_{R_0}$. By Lemma \ref{lem:key estimate for CSP} (ii), there exists $x_1\in (x_0+r_0,x_0+2r_0)$ such that 
\begin{equation*}
\begin{aligned}
&\dP\left( \int_0^{\tau}  \sigma(u(s,x_1))   ds \geq p_1 \right) 
\leq \dP\left(  \int_0^{\tau}  \sigma(u(s,x_0))    ds \geq p_{0} \right) + C r_0^{-\frac{3\beta}{2}}\left(\frac{p_0^{1+\lambda}}{p_{1}\sigma(p_0^\lambda)}\right)^{\beta},
\end{aligned} 
\end{equation*}
where $C = C(\beta,\lambda, \gamma,H)>0$. Now, we proceed the iteration as follows. For $k\geq 2 $, take $x_{k-1}$ as a new origin and apply Lemma \ref{lem:key estimate for CSP} (ii) to find $x_k\in (x_{k-1}+r_{k-1},x_{k-1}+2r_{k-1})$ satisfying
\begin{equation}
    \label{iteration}
\begin{aligned}
&\dP\left( \int_0^{\tau}  \sigma(u(s,x_{k})) ds \geq p_k \right) 
\leq \dP\left(  \int_0^{\tau}  \sigma(u(s,x_{k-1}))  ds \geq p_{k-1} \right) + Cr_{k-1}^{-\frac{3\beta}{2}}\left(\frac{p_{k-1}^{1+\lambda}}{p_{k}\sigma(p_{k-1}^\lambda)}\right)^{\beta}.
\end{aligned}
\end{equation} Since $|x_k-x_{k-1}|\leq 2r_{k-1} = k^{-1-c/3}$ for all $k\geq 1 $, there exists $y$ such that $x_k \rightarrow y$ as $k\rightarrow \infty$. From \eqref{iteration}, by summing over $k=1,2,\dots$, we have  
\begin{equation}\label{eq: sum of iteration}
\begin{aligned}
\dP\left( \int_0^{\tau} \sigma(u(s,y)) ds > 0 \right) \leq \dP\left(  \int_0^{\tau} \sigma(u(s,x_0)) ds \geq p_0 \right) + \fR(m),
\end{aligned}
\end{equation}
where
\begin{equation*}
\begin{aligned}
\fR(m)
&:= C\sum_{k=1}^\infty \left( r_{k-1}^{-\frac{3}{2}} \frac{p_{k-1}^{1+\lambda}}{p_{k}\sigma(p_{k-1}^\lambda)}\right)^{\beta}.
\end{aligned}
\end{equation*} 

Now note that \eqref{eq:csp_condition} implies that there exist constants $\alpha>\frac{5}{2}$ and $k_0(\alpha)\geq 1 $ such that for all $k\geq k_0$ 
\begin{equation}\label{eq:decay of nonlinearity for CSP}
  \frac{e^{-k}}{\sigma(e^{-k})} \leq k^{-\alpha}.
 \end{equation} 
 On the other hand, observe that 
 \begin{equation*}
  \frac{p_{k-1}}{p_{k}} = e^{1/\lambda}
 \end{equation*} for all $k\geq 1$. Therefore, we can choose $c=c(\alpha)>0$ such that for all $k\geq k_0$
\begin{equation*}
  \frac{p_{k-1}^{1+\lambda}}{p_{k}\sigma(p_{k-1}^\lambda)} \leq \frac{e^{1/\lambda}}{( k-1+  m )^{5/2+c}},
\end{equation*} where we applied \eqref{eq:decay of nonlinearity for CSP}. 
Note that $k_0$ can be chosen independently of $m\geq 1 $. Thus, we have an upper bound of $\fR(m)$:
\begin{equation*}
  \fR(m) \leq C\sum_{k=1}^\infty \left(\frac{e^{1/\lambda}k^{(3+c)/2}}{(k-1+ m)^{5/2+c}}\right)^\beta.
\end{equation*} In addition, for any $c>0$ and $\beta\in(0,1)$, we have  
\begin{equation*}
  \left( \left( \frac{5}{2} +c \right) - \frac{3+c}{2} \right)^\beta >1.
\end{equation*} Now it follows that for all sufficiently large $m$, we have $\fR(m)<\infty$.  Moreover, it is easy to see that  $\lim_{m\to \infty} \fR(m) =0$ by the dominated convergence theorem. We point out that \eqref{eq: sum of iteration} also holds for any choice of $x_0 = R > R_0 +1$ with the same constants $C>0$ and $y>0$ which are independent of $R$. In other words, for any $R>R_0 +1$, we have
\begin{equation*}
\dP\left( \int_0^\tau \sigma(u(s,R+y))ds > 0  \right) 
\leq \dP\left( \int_0^\tau \sigma(u(s,R))ds  >  e^{-m/\lambda} \right) + \fR(m),
\end{equation*} since $p_0 = e^{-m/\lambda}$. We now send $R\to \infty$ to get that 
\begin{equation}
\label{eq:last ineq}
\limsup_{R\to\infty}\dP\left( \int_0^\tau \sigma(u(s,R))ds > 0  \right) 
\leq \limsup_{R\to\infty}\dP\left( \int_0^\tau \sigma(u(s,R))ds  > e^{-m/\lambda} \right) + \fR(m).
\end{equation} Regarding the first term on the right-hand side of \eqref{eq:last ineq}, by combining Lemma \ref{lem:property of sigma} and Lemma \ref{lem:limit of the time L1 norm}, we have for any $\delta \in (0, \infty)$,
\begin{equation*}
\begin{aligned}
    \limsup_{R\to\infty}\dE \left[  \int_0^\tau \sigma(u(s,R)) ds \right] &\leq \limsup_{R\to\infty}\dE \left[  \int_0^\tau \left[ \left( L_\fd \vee \delta^{-1} \sigma(\delta) \right)u(s,R) + \sigma(\delta) \right]ds \right]\\
    &\leq  \dE\left[ \int_0^\tau \sigma(\delta )ds \right].
\end{aligned}
\end{equation*} Letting $\delta \to 0$ in the last line and applying Chebyshev's inequality yields
\begin{equation*}
     \limsup_{R\to\infty}\dP\left( \int_0^\tau \sigma(u(s,R))ds  >e^{-m/\lambda} \right)  =0.
      \end{equation*} Thus, by taking $m\to \infty$ in \eqref{eq:last ineq} and recalling that $\lim_{m\to \infty} \fR(m) = 0$, we get 
\begin{equation*}
  \limsup_{R\to\infty}\dP\left( \int_0^\tau \sigma(u(s,R))ds  >  0  \right) =0,
\end{equation*} which completes the proof.

\end{proof}


\subsection{Proof of Lemma \ref{lem:key estimate for CSP}}\label{subsec:proof of props}

In this section, we prove Lemma \ref{lem:key estimate for CSP}, which is a key ingredient in the proof of Theorem \ref{thm:cpt_support}. The following lemma presents the estimation of the quadratic variation outside the initial support in the weak formulation \eqref{eq:sol_int_eq_form}. This estimation is obtained using a suitable test function and is controlled by the time integral of the solution $u$ at the boundary. The result of this lemma will be used later in the proof of Lemma \ref{lem:key estimate for CSP}.

\begin{lemma}
\label{prop:estimation of quadratic variation} Let $T>0$ and let $\tau\leq T$ be a bounded stopping time. Let $u\in C([0,\tau];C_{tem}(\dR))$ be a nonnegative weak solution to \eqref{eq:SHE} satisfying the assumptions in Theorem~\ref{thm:cpt_support}. Let $\eta\leq \tau$ be a bounded stopping time and let $R>R_0 \vee 1 $ where $R_0>0$ satisfies $\supp(u_0)\subseteq I_{R_0}$. Then, for every  $\beta\in(0,1)$, there exists $C= C(\beta) >0$ such that 
\begin{equation}
\label{eq:estimation of quadratic variation}
\begin{aligned}
&\dE\left[\left( \int_0^\eta\int_R^\infty (x-R)^2|\sigma(u(s,x))|^2 dxds \right)^{\beta/2} \right]
\leq C\dE\left[\left(\int_0^\eta u(s,R) ds\right)^\beta\right].\\
\end{aligned}
\end{equation}
\end{lemma}
\begin{proof} The proof of this lemma follows the same arguments as those used in the proof of \cite[Lemma 4.1]{han2023compact} or \cite[Lemma 2.1]{krylov1997result}. However, for the sake of completeness, we provide a sketch of the proof here.

For each $n \geq 1$, let $\psi \in C^\infty_c(\dR)$ be a nonnegative even function with $\psi(0) = 1$, and define $\psi_n(x) := \psi((x-R)/n)$. Then, by choosing the test function $\phi(x) := (x-R)_+ \psi_n(x)$ in the definition of a weak solution $u$ (see \eqref{eq:sol_int_eq_form}), we obtain
\begin{equation}\label{eq:weak formulation of u with test function x psi}
  \begin{aligned}
    0\leq \int_{R}^\infty (x-R) \psi_n(x) u(t,x) dx = \int_0^t u(s,R) ds + \fA^n_t + \fM^n_t,
  \end{aligned}
\end{equation} where 
\begin{equation*}
  \begin{aligned}
    \fA^n_t : = \int_0^t \int_R^\infty u(s,x)\cdot \frac{1}{2}\partial^2_{xx}(x\psi_n(x)) dxds=\int_0^t \int_R^\infty u(s,x) \left[\frac{\psi'_n(x) + (1+x) \psi''_n(x)}{2}\right] dx,
  \end{aligned}
\end{equation*} and $\fM^n_t$ is a local martingale with quadratic variation
\begin{equation*}
  \begin{aligned}
    \langle \fM^n\rangle_t = \int_0^t \int_R^\infty (x-R)^2 \psi_n^2(x) \sigma^2(u(s,x)) dxds.
  \end{aligned}
\end{equation*} Note that $\phi$ belongs to $C_c(\dR)$, but not in $\cS$. However, \eqref{eq:weak formulation of u with test function x psi} can be justified by a simple approximation argument (see for instance \cite[Lemma 3.1]{krylov1997result}).  Observe that $|\psi'_n(x)| \lesssim n^{-1}$ and $|\psi''_n(x)| \lesssim n^{-2}$ for all $x\in\dR$. Thus, we have that almost surely
\begin{equation}\label{eq:limit of A^n}
  \begin{aligned}
    \lim_{n\to \infty } \sup_{t\leq T}|\fA^n_t| \lesssim \lim_{n\to \infty} \frac{1}{n}\left(  \sup_{t\leq T} \int_R^\infty (x-R) u(t,x) dx\right) =0. 
  \end{aligned}
\end{equation} Here we used the fact that the supremum of the integral on the RHS is finite owing to \cite[Lemma 5.1]{han2023compact}. We have obtained that for all $t\leq T$
\begin{equation*}
 0\leq \int_0^t u(s,R) ds + |\fA^n_t| + \fM^n_t. 
\end{equation*} This implies that for any stopping time $\bar{\eta}\leq \eta$,
\begin{equation*}
  (\fM^n_{\bar{\eta}})_- \leq \int_0^{\bar{\eta}} u(s,R) ds + |\fA^n_{\bar{\eta}}|. 
\end{equation*} Since $\fM^n_t$ is a local martingale, for any fixed $n$, we can choose a sequence of stopping times $\{\eta^i\}_{i=1}^\infty$ such that $\lim_{i\to \infty} \eta^i =\infty$ and $\fM^n_{t\wedge \eta^i}$ is a martingale on $[0,T]$ for each $i\in \dN$. Then, by the optional stopping time, $\fM^n_{t\wedge \eta^i \wedge \eta}$ is a martingale and $\dE [ \fM^n_{t\wedge \eta^i \wedge {\bar{\eta}}}]=0$ for any $i$ and $t\in[0,T]$. This shows that $\dE [ (\fM^n_{t\wedge \eta^i \wedge {\bar{\eta}}})_+]=\dE [ (\fM^n_{t\wedge \eta^i \wedge {\bar{\eta}}})_-]$. Since $u$ is nonnegative, we know that $\int_0^t u(s,R) ds + |\fA^n_t| $ is nondecreasing in $t\in[0,T]$, which ensures that 
\begin{equation*}
  \dE\left[|\fM^n_{ \eta^i \wedge \bar{\eta}}|  \right] \leq 2 \dE \left[ \bar{\fA}^n_{\bar{\eta}}\right],
\end{equation*} where 
\begin{equation*}
  \bar{\fA}^n_{\bar{\eta}} := \int_0^{\bar{\eta}} u(s,R) ds + |\fA^n_{\bar{\eta}}|. 
\end{equation*} By Fatou's lemma, we further obtain that 
\begin{equation*}
  \dE\left[|\fM^n_{ {\bar{\eta}}}|  \right] \leq 2 \dE \left[ \bar{\fA}^n_{\bar{\eta}} \right].
\end{equation*} Since ${\bar{\eta}}\leq \eta$ is arbitrary, we then apply \cite[Theorem III.6.8]{diffusion} to obtain that for any $\beta\in(0,1)$ and any stopping time $\eta\leq T$,
\begin{equation*}
  \dE\left[\sup_{\eta\leq T}  |\fM^n_\eta| ^\beta\right] \leq C_1 \dE \left[ \sup_{\eta\leq T} |\bar{\fA}^n_\eta|^\beta  \right]
\end{equation*} for some constant $C_1= C_1(\beta)>0$. Therefore, by the BDG inequality and letting $n\to \infty,$  there exists $C=C(\beta)>0$ such that 
\begin{equation*}
  \dE\left[\left( \int_0^\eta\int_R^\infty (x-R)^2|\sigma(u(s,x))|^2 dxds \right)^{\beta/2} \right]
\leq C \cdot \limsup_{n\to \infty} \dE \left[ \sup_{\eta\leq T} |\bar{\fA}^n_\eta|^\beta  \right]= C \dE\left[\left(\int_0^\eta u(s,R) ds\right)^\beta\right],
\end{equation*} where the last equality follows from \eqref{eq:limit of A^n}.
\end{proof}

We now present the proof of Lemma \ref{lem:key estimate for CSP}, which is analogous to Lemma 2.1 in \cite{krylov1997result}, as mentioned in Section \ref{sec: introduction}. However, our proof differs from Krylov's \cite{krylov1997result} since we consider a broader class of $\sigma$ than the ones in \cite{krylov1997result} (specifically, $\sigma(u)=u^\gamma$ for $\gamma \in (0, 1)$). In particular, we use a novel approach to derive \eqref{eq:controlling int of u(s,0)}, rather than applying Jensen's inequality as in \cite{krylov1997result}, which relies on partitioning the time interval where the solution remains not too small.

\begin{proof}[\textbf{Proof of Lemma \ref{lem:key estimate for CSP}}]

\textit{Proof of Lemma \ref{lem:key estimate for CSP} (i)}. We only prove that the probability of the LHS is bounded above by the one of the RHS since the proof of the other direction follows immediately from the fact that $u$ is continuous almost surely. Let $p,q\in(0,\fd\wedge 1 )$, $\beta \in (0,1)$ and $R>R_0\vee 1 $. Define a stopping time 
\begin{equation}
\label{eq:definition of kappa}
\kappa := \tau\wedge\inf\left\{ t\geq0: \int_0^t u(s,R) ds \geq q \right\}.
\end{equation}
By Chebyshev's inequality and \eqref{eq:estimation of quadratic variation} with the stopping time $\eta=\kappa$, we have
\begin{equation}
\label{eq:ineq to prove maximum principle}
\begin{aligned}
\dP&\left( \int_0^\tau \int_R^\infty (x-R)^2|\sigma(u(s,x)))|^2 dxds \geq p \right) 
\\&\leq  \dP\left( \kappa < \tau \right) + \dP\left( \int_0^\kappa \int_R^\infty (x-R)^2|\sigma(u(s,x)))|^2 dxds \geq p \right) \\
&\leq \dP\left( \int_0^\tau u(s,R) ds \geq q \right) + p^{-\beta/2}C\dE\left[\left(\int_0^\kappa  u(s,R)  ds\right)^{\beta}\right] \\
&\leq \dP\left( \int_0^\tau u(s,R) ds \geq q \right) + C\frac{q^\beta}{p^{\beta/2}}
\end{aligned}
\end{equation}
where $C = C(\beta)>0$ is the constant given in Lemma \ref{prop:estimation of quadratic variation} and we have used the definition of $\kappa$ in \eqref{eq:definition of kappa} in the second inequality. By letting $q\downarrow0$ and then $p\downarrow0$, we have
\begin{equation*}
\begin{aligned}
\dP\left( \int_0^\tau \int_R^\infty (x-R)^2|\sigma(u(s,x)))|^2 dxds >0  \right) 
\leq \dP\left( \int_0^\tau u(s,R) ds > 0 \right).
\end{aligned}
\end{equation*}
Thus,
\begin{equation*}
\begin{aligned}
\dP\left( \int_0^\tau u(s,R) ds = 0 \right) 
&\leq \dP\left( \int_0^\tau \int_R^\infty (x-R)^2|\sigma(u(s,x)))|^2 dxds = 0  \right) \\
&\leq\dP \left( u(s,x) =0 \quad \text{for all }s\in[0,\tau] \text{ and } x > R\right).
\end{aligned}
\end{equation*} The proof is completed.

\vspace{2mm}

\textit{Proof of Lemma \ref{lem:key estimate for CSP} (ii)}.  Let us fix $\beta \in (0,1)$, and $p,q\in(0,\fd^{1/\lambda} \wedge 1 )$. For $R>R_0\vee 1$,  we define  a stopping time $\kappa$ as 
\begin{equation}\label{eq:definition of kappa 2}
\kappa := \tau\wedge\inf\left\{ t\geq0: \int_0^t \sigma(u(s,R)) ds \geq q \right\},
\end{equation} which differs from the previous definition of $\kappa$ in \eqref{eq:definition of kappa}. Let $ r>0$ and note that 
\begin{equation*}
\begin{aligned}
&r^{-1}\int_{R+r}^{R+2r} \dP\left( \int_0^\tau \sigma(u(s,x)) ds \geq p \right)dx 
\leq  \dP(\kappa<\tau) + r^{-1}\int_{R+r}^{R+2r} \dP\left( \int_0^{\kappa}  \sigma(u(s,x)) ds \geq p \right)dx. \\
\end{aligned} 
\end{equation*} It follows from the definition of $\kappa$ (see \eqref{eq:definition of kappa 2}) that 
\begin{equation}
\label{eq:upper bound using kappa}
\begin{aligned}
r^{-1}&\int_{R+r}^{R+2r} \dP\left( \int_0^\tau \sigma(u(s,x)) ds \geq p \right)dx \\
&\leq  \dP\left(\int_0^\tau \sigma(u(s,R))ds \geq q  \right) + r^{-1}\int_{R+r}^{R+2r} \dP\left( \int_0^{\kappa}  \sigma(u(s,x)) ds \geq p \right)dx. 
\end{aligned} 
\end{equation} Then, by Chebyshev's inequality, \eqref{eq:estimation of quadratic variation} with $\eta=\kappa$, and Jensen's inequality, we have 
\begin{equation}
\label{eq:proof of main ineq}
\begin{aligned}
r^{-1}\int_{R+r}^{R+2r} &\dP\left( \int_0^{\kappa}  \sigma(u(s,x)) ds \geq p\right)dx \\
&\leq r^{-1}p^{-\beta}\int_{R+r}^{R+2r} \dE\left[\left(\int_0^\kappa \sigma(u(s,x)) ds  \right)^\beta\right] dx \\
&\leq  C r^{-\frac{3\beta}{2}}p^{-\beta} \dE\left[\left(\int_0^\kappa\int_R^\infty (x-R)^2|\sigma(u(s,x))|^2 dxds  \right)^{\beta/2}\right] \\
&\leq  C r^{-\frac{3\beta}{2}}p^{-\beta} \dE\left[\left(\int_0^\kappa u(s,R) ds  \right)^{\beta}\right].
\end{aligned} 
\end{equation} for some constant $C=C(T, \beta) >0$. Next, we claim that there exists $C=C(\gamma, H)>0$ such that for all $\omega\in \Omega$ satisfying \eqref{eq:bound for the Holder coefficient of u}, we have  
\begin{equation}
\label{eq:controlling int of u(s,0)}
\int_0^{\kappa}    u(s,R) ds  \leq C \frac{q^{1+\lambda}}{\sigma(q^\lambda)}.
\end{equation} If this claim holds, by substituting \eqref{eq:controlling int of u(s,0)} into \eqref{eq:proof of main ineq}, 
\begin{equation*}
  \begin{aligned}
    r^{-1}\int_{R+r}^{R+2r} \dP\left( \int_0^{\kappa}  \sigma(u(s,x)) ds \geq p\right)dx \leq C r^{-\frac{3\beta}{2}} \left(  \frac{q^{1+\lambda}}{p\sigma(q^\lambda)}\right)^\beta,
  \end{aligned}
\end{equation*}  where $C= C(\beta,\gamma, H,T )>0$. Therefore, from \eqref{eq:upper bound using kappa}, we can deduce that 
\begin{equation}\label{eq:inequality for mean value theorem}
\begin{aligned}
r^{-1}&\int_{R+r}^{R+2r} \dP\left( \int_0^\tau \sigma(u(s,x)) ds \geq p \right)dx \\
&\leq  \dP\left(\int_0^\tau \sigma(u(s,R))ds \geq q  \right) + C r^{-\frac{3\beta}{2}} \left(  \frac{q^{1+\lambda}}{p\sigma(q^\lambda)}\right)^\beta.
\end{aligned} 
\end{equation} Applying the mean value theorem into \eqref{eq:inequality for mean value theorem}, we find a $x\in(r,2r)$ to complete the proof of \eqref{eq:ito inequality for CSP}.

For the remainder of the proof, we will prove the claim \eqref{eq:controlling int of u(s,0)}. Recall that  $q<\fd^{1/\lambda} \wedge 1$ and $\lambda = \gamma/(1+\gamma)\in (0,1)$. Therefore, by Assumption \ref{assumption:basic condition for sigma}, 
\begin{equation}\label{eq:bound for u smaller than q^lambda}
\begin{aligned}
\int_0^{\kappa} u(s,R)\1_{u(s,R)\leq q^\lambda} ds 
\leq \frac{q^\lambda}{\sigma(q^\lambda)}\int_0^{\kappa} \sigma(u(s,R)) ds \leq \frac{q^{1+\lambda}}{\sigma(q^\lambda)}.
\end{aligned}
\end{equation}
Additionally, by \eqref{eq:definition of kappa} and the fact that $u\mapsto \sigma(u)$ is nondecreasing, we have
\begin{equation*}
\begin{aligned}q&\geq \int_0^{\kappa} \sigma(u(s,R)) \1_{u(s,R) > q^\lambda} ds 
\geq\sigma(q^\lambda)\int_0^{\kappa}  \1_{u(s,R) > q^\lambda} ds.
\end{aligned}
\end{equation*}
Therefore, we obtain
\begin{equation}
\label{eq:controlling int of u(s,0)2}
\int_0^{\kappa} \1_{u(s,R) > q^\lambda} ds \leq \frac{q}{\sigma(q^\lambda)}.
\end{equation}On the other hand, by the continuity of $u$, we can decompose the open set $\{s\in (0,\kappa) : u(s,R) >q^\lambda\}$ as 
\begin{equation*}
\begin{aligned}
\left\{ s \in (0,\kappa): u(s,R) > q^\lambda \right\} = \dot\cup_{j}(s_q^{j},t_q^{j}), \\
\end{aligned}
\end{equation*} where $\dot\cup$ denotes the disjoint union. Observe that again by the continuity of $u(\cdot,R)$, we have $u(s_q^j,R) = q^\lambda$. We employ \eqref{eq:bound for the Holder coefficient of u} to proceed as 
\begin{equation*}
  \begin{aligned}
    \int_0^{\kappa} u(s,R) \1_{u(s,R) > q^\lambda} ds
&= \sum_{j}\int_{s_q^j}^{t_q^{j}} u(s,R) ds \\
&= \sum_{j}\left[ \int_{s_q^j}^{t_q^{j}} u(s,R) - u(s_q^j,R) ds + \int_{s_q^j}^{t_q^{j}} u(s_q^j,R) ds \right] \\
&\leq C\sum_{j} \int_{s_q^j}^{t_q^{j}} |s - s_q^j|^\gamma ds + q^\lambda\int_0^{\kappa} \1_{u(s,R)>q^\lambda}  ds  \\
&\leq C\sum_{j}  |t_q^j - s_q^j|^{1+\gamma} + q^\lambda\int_0^{\kappa} \1_{u(s,R)>q^\lambda}  ds.
  \end{aligned}
\end{equation*} for some constant $C=C(\gamma,H)>0$. Now apply \eqref{eq:controlling int of u(s,0)2} several times to yield 
\begin{equation}\label{eq:bound for u bigger than q^lambda}
\begin{aligned}
\int_0^{\kappa} u(s,R) \1_{u(s,R) > q^\lambda} ds
&\leq C\sum_{j}  \left| \int_{s_q^j}^{t_q^j} ds \right|^{1+\gamma} + \frac{q^{1+\lambda}}{\sigma(q^\lambda)} \\
&\leq C \left| \sum_{j} \int_{s_q^j}^{t_q^j} ds \right|^{1+\gamma} + \frac{q^{1+\lambda}}{\sigma(q^\lambda)}  \\
&\leq C \left|  \int_0^{\kappa} \1_{u(s,R)>q^\lambda} ds \right|^{1+\gamma} + \frac{q^{1+\lambda}}{\sigma(q^\lambda)}  \\
&\leq C \left(\frac{q}{\sigma(q^\lambda)}\right)^{1+\gamma} + \frac{q^{1+\lambda}}{\sigma(q^\lambda)}. \\
\end{aligned}
\end{equation} Recalling that $\lambda= \gamma/(1+\gamma)$, we see that $1-\lambda \gamma^{-1} > \lambda $, which in turn implies that 
\begin{equation}\label{eq:upper bound using lambda}
  \left(\frac{q}{\sigma(q^\lambda)}\right)^{1+\gamma} = \frac{q^{1+\lambda}}{\sigma(q^\lambda)} \left( \frac{q^{1-\lambda \gamma^{-1}}}{\sigma(q^\lambda)}\right)^\gamma \leq \frac{q^{1+\lambda}}{\sigma(q^\lambda)} \left( \frac{q^{\lambda}}{\sigma(q^\lambda)}\right)^\gamma,
  \end{equation} where we used that $q<1$. Moreover, due to the monotonicity of $u\mapsto u/\sigma(u)$, we can find a constant $C_1=C_1(\gamma)>0$ such that for all $q\in(0,\fd^{1/\lambda}\wedge 1 )$ 
  \begin{equation*}
  \left( \frac{q^{\lambda}}{\sigma(q^\lambda)}\right)^\gamma \leq C_1.
  \end{equation*} Therefore, by \eqref{eq:upper bound using lambda}, we have shown that 
  \begin{equation*}
    \left(\frac{q}{\sigma(q^\lambda)}\right)^{1+\gamma} \leq C_1 \frac{q^{1+\lambda}}{\sigma(q^\lambda)}.
  \end{equation*} Combining the above inequality with \eqref{eq:bound for u bigger than q^lambda}, we prove that there exists $C=C(\gamma, H)$ such that 
\begin{equation}\label{eq:bound for u bigger than q^lambda 2}
  \int_0^\kappa u(s,R)\1_{u(s,R)>q^\lambda} ds \leq C \frac{q^{1+\lambda}}{\sigma(q^\lambda)}.
\end{equation} By \eqref{eq:bound for u smaller than q^lambda} and \eqref{eq:bound for u bigger than q^lambda 2}, we can conclude that \eqref{eq:controlling int of u(s,0)} holds uniformly over all $q\in(0,\fd^{1/\lambda}\wedge 1 )$. 
\end{proof}






\begin{small}

\noindent\textbf{Beom-Seok Han} [\texttt{b\_han@sungshin.ac.kr}]\\
\noindent Sungshin Women's University, Seoul, South Korea\\

\noindent\textbf{Kunwoo Kim} [\texttt{kunwoo@postech.ac.kr}]\\
\noindent Pohang University of Science and Technology (POSTECH), Pohang, Gyeongbuk, South Korea \\

\noindent\textbf{Jaeyun Yi} [\texttt{jaeyun.yi@epfl.ch}]\\
\noindent \'Ecole Polytechnique F\'ed\'erale de Lausanne (EPFL), Lausanne, Switzerland

\end{small}

\end{document}